%% file: arxiv_2.tex
\documentclass[11pt]{article}

\usepackage{geometry}
\usepackage{mathtools}
\usepackage{appendix}
\usepackage{tikz}
\usetikzlibrary{arrows}
\usetikzlibrary{patterns}
\usetikzlibrary{hobby}
\usetikzlibrary{decorations.pathreplacing}

\usepackage{graphicx}
\usepackage{esint}
\usepackage{amsmath,amssymb,amsthm}
\usepackage{enumerate}
\usepackage{hyperref}
\usepackage[active]{srcltx}
\usepackage[T1]{fontenc}
\usepackage{color}
\usepackage[makeroom]{cancel}

\newtheorem{theo}{Theorem}

\newtheorem{lem}{Lemma}[section]
\newtheorem{defi}[lem]{Definition}

\newtheorem{prop}[lem]{Proposition}

\newtheorem{conj}[theo]{Conjecture}
\newtheorem*{main_result}{Main result}

\newcommand{\eps}{\varepsilon}
\newcommand{\qtext}[1]{\quad\mbox{#1}\quad}
\newcommand{\qqtext}[1]{\qquad\mbox{#1}\qquad}

\newcommand{\rd}{\mathrm d}
\newcommand{\R}{\mathbb{R}}

\newcommand{\N}{\mathbb{N}}
\newcommand{\narrowcv}{\overset{\ast}{\rightharpoonup}}
\newcommand{\grad}{\operatorname{grad}}

\newcommand{\M}{\mathcal{M}}
\newcommand{\W}{{\mathcal{W}}}
\newcommand{\E}{{\mathcal{E}}}
\newcommand{\I}{{\mathcal{I}}}
\newcommand{\F}{{\mathcal{F}}}
\newcommand{\D}{{\mathcal{D}}}
\newcommand{\X}{{\mathcal{X}}}
\newcommand{\Dom}{\operatorname{Dom}}
\newcommand{\tr}{\operatorname{tr}}
\renewcommand{\H}{{\mathcal{H}}}
\newcommand{\argmin}{\operatorname{Argmin}}

\newcommand{\dist}{\mathsf d}

\newcommand{\Oed}{{\Omega_{\varepsilon,\delta}}}
\newcommand{\Gd}{{\Gamma_\delta}}
\newcommand{\ologM}{{\overline H'_M}}
\newcommand{\ulogM}{{\underline H'_M}}

\newcommand{\mrest}{
  \,\raisebox{-.127ex}{\reflectbox{\rotatebox[origin=br]{-90}{$\lnot$}}}\,%
}
\newcommand{\pf}{{}_{\#}}

\renewcommand{\P}{\mathcal{P}}
\renewcommand{\o}{\omega}
\newcommand{\g}{\gamma}
\renewcommand{\O}{\Omega}
\newcommand{\G}{\Gamma}
\newcommand{\pO}{{\partial\Omega}}
\newcommand{\LO}{{\mathcal L_{\Omega}^d}}
\newcommand{\LG}{{\mathcal L_{\Gamma}^{d-1}}}
\newcommand{\bO}{{\overline\Omega}}
\renewcommand{\L}{\mathcal L}
\newcommand{\J}{\mathcal J}
\newcommand{\orho}{{\overline\rho}}
\newcommand{\oo}{{\overline\omega}}
\newcommand{\og}{{\overline\gamma}}
\newcommand{\trho}{{\tilde\rho}}

\def\dive{\operatorname{div}}

\numberwithin{equation}{section}
\begin{document}
\title{Sticky-reflecting diffusion as a Wasserstein gradient flow}
\date{}
\author{Jean-Baptiste Casteras, L\'eonard Monsaingeon\footnote{Corresponding author}, Filippo Santambrogio}

\maketitle
\abstract{
In this paper we identify the Fokker-Planck equation for (reflected) Sticky Brownian Motion as a Wasserstein gradient flow in the space of probability measures.
The driving functional is the relative entropy with respect to a non-standard reference measure, the sum of an absolutely continuous interior part plus a singular part supported on the boundary.
Taking the small time-step limit in a minimizing movement (JKO scheme) we prove existence of weak solutions for the coupled system of PDEs satisfying in addition an Energy Dissipation Inequality.
\\
\begin{center}
{\bf R\'esum\'e}
\end{center}

Dans cet article l'\'equation de Fokker-Planck pour le Sticky mouvement Brownien (r\'efl\'echi) est identifi\'ee comme un flot-gradient Wasserstein dans l'espace des mesures de probabilit\'e.
La fonctionnelle pilotant la dynamique est l'entropie relative \`a une mesure de r\'ef\'erence non standard, somme d'une partie int\'erieure absolument continue et d'une partie singuli\`ere support\'ee au bord.
En prenant la limite de petit pas de temps dans un mouvement minimisant (sch\'ema JKO) on prouve l'existence de solutions faibles pour le syst\`eme coupl\'e d'EDPs satisfaisant de plus une In\'egalit\'e de Dissipation de l'\'Energie.
}
\\

\section{Introduction}
Since the work of R. Jordan, D. Kinderlehrer, and F. Otto \cite{JKO98} it is now well understood that the classical Fokker-Planck equation
$ \partial_t\rho =\Delta\rho+\dive(\rho\nabla V)$ in bounded domains $\O\subset \R^d$ can be reinterpreted as the Wasserstein gradient flow $\frac{d\rho}{dt}=-\grad_\W\H(\rho|\mu)$ of the relative entropy in the space of probability measures, where $\mu=\frac 1{\mathcal Z}e^{-V}\in \P(\O)$ is the stationary Gibbs measure associated with the background potential $V:\Omega\to \R$.
Here $\W$ is the quadratic Wasserstein distance over the smooth domain $\O\subset \R^d$, and $\H(\rho|\mu)=\int_\O \frac{\rho}{\mu}\log\left(\frac\rho\mu\right)\rd\mu$ is the relative entropy (Kullback-Leibler divergence) of $\rho$ w.r.t. the reference measure $\mu$.
This has shed a whole new light on variational evolution of probability measures as gradient flows and the theory now covers advection-diffusion-aggregation equations \cite{carrillo2003kinetic,blanchet2008convergence}, Porous Medium Equation \cite{otto2001geometry} and doubly nonlinear parabolic equations \cite{agueh}, fourth-order quantum drift-diffusion \cite{matthes2009family}, reaction-diffusion equations with mass variations \cite{kondratyev2016new,GLM19,di2020tumor}, and many more.
We refer to the by-now classical textbooks \cite{villani_BIG,villani2003topics,OTAM} as well as to the survey \cite{santambrogio2017euclidean} for further discussions, references, and applications of this steadily growing topics.

A common feature of all the aforementioned models and equations is that they deal with measures $\rho=\rho(x)\cdot \rd x$ which are absolutely continuous with respect to either the Lebesgue measure in Euclidean domains, or its equivalent - the volume form - in Riemannian manifolds.
For the standard Fokker-Planck equation this is clearly legitimate due to strong regularizing properties of the Laplacian, the generator of the standard (reflected) Brownian Motion with stationary measure precisely given by the Lebesgue measure.
Recently however, the so-called (reflected) \emph{Sticky Brownian Motion} \cite{feller1952parabolic} (SBM in short) started attracting renewed interest \cite{peskir2015boundary,engelbert2014stochastic,bou2020sticky,konarovskyi2021spectral}.
Roughly speaking, SBM is a $\R^d$-valued stochastic process which behaves as a standard diffusion as long as it remains in the interior of a prescribed domain $\Omega\subset \R^d$.
When it hits the boundary $\G=\pO$ it sticks there for a random positive amount of time, while following an intrinsic tangential diffusion (generated by the Laplace-Beltrami operator $\Delta_\G$ thereon).
Eventually SBM almost surely reenters the domain and resumes its standard Brownian behaviour, until hitting the boundary again, and so on.
We refer to \cite{grothaus2017stochastic,fattler2016construction} for the construction and regularity properties of SBM via Dirichlet-forms.
In functional analytic terms, SBM is best described by its generator
$$
\L f(x)=\lim\limits_{t\downarrow 0}\frac{\mathbb E f(X_t)-f(x)}{t}
=
\begin{cases}
\frac 12\Delta f(x) & \mbox{if }x\in \O\\
\frac{a}{2}\Delta_\G f(x)-b\partial_\nu f(x) & \mbox{if }x\in\G=\pO
\end{cases}
$$
were the parameters $a,b>0$ are the tangential diffusivity and stickiness of the SBM process, respectively.
The stationary measure accordingly comprises an interior $\LO$ and boundary $\LG$ Lebesgue measures, and the laws of motion couple the interior domain and its boundary through an exchange term.
The relevant probability distributions thus cannot simply be absolutely continuous, and we write throughout
$$
\rho= \o+\g\in \P(\bO)
\qqtext{with}
\begin{cases}
\o=\rho\mrest\O\in \M^+(\O)\\
\mbox{and}\\
\g=\rho\mrest\G\in\M^+(\G)
\end{cases}
$$
for the associated interior/boundary decomposition $\bO=\O\cup\G$.
Discarding the probabilistic $\frac 12$ factors for convenience, the abstract Fokker-Planck equation $\partial_t\rho=\L^*\rho$ for pure SBM reads here
\begin{equation}
\label{eq:sticky_fokker_planck_ab}
\begin{cases}
  \partial_t\omega=\Delta\omega & \mbox{in }\Omega\\
    \omega=b\gamma & \mbox{on }\partial\Omega\\
  \partial_t\gamma=a\Delta_\G\gamma-\partial_\nu\omega & \mbox{in }\G
\end{cases}.
\end{equation}
In terms of PDEs this is a coupled system of bulk/interface diffusions, and the $\partial_\nu\o$ exchange term corresponds at the stochastic level to the jump rate between the two interior/boundary behaviours.
Here we focus on the case $a=b=1$, and in this paper we shall make a case that this is again a gradient flow
\begin{equation}
\label{eq:sticky_fokker_planck_GF}
\begin{cases}
 \partial_t\omega=\Delta\omega & \mbox{in }\Omega\\
 \omega=\gamma & \mbox{on }\partial\Omega\\
 \partial_t\gamma=\Delta_\G\gamma-\partial_\nu\omega  & \mbox{in }\G
\end{cases}
\quad\iff\quad
\frac{d\rho}{dt}=-\grad_{\W}\H(\rho\,\vert\,\mu)
\qtext{for}
\mu\coloneqq \LO+\LG.
\end{equation}
Perhaps surprisingly, the transportation distance involved here is still the classical, quadratic Wasserstein distance $\W$ over $\bO$.
This is due to our crucial assumption that $a=1$, for which the interior $\Delta$ and boundary $\Delta_\G$ diffusions are tangentially matched.
For $a\neq 1$ the transportation distance must be adapted \cite{CMN25} and the gradient-flow structure will be investigated in a future work \cite{CMN24}.
The case of coefficients $b\neq 1$ can however be treated by simply adapting the reference measure $\mu_b=b\LO+\LG$ in the driving functional $\rho\mapsto\H(\rho\,\vert\,\mu_b)$, and we chose to focus on $b=1$ only to clarify the exposition.
Similarly, background potentials $V_\O,V_\G$ could very well be included in the total free energy $\F=\H+\int _\O V_\O\rd\o+\int_\G V_\G\rd \g$ in order to account for drift in the Fokker-Planck equation, but we simply ignore this possibility in order not to overburden the analysis.
Note that a rather complete mathematical analysis (existence, uniqueness, regularity) of \eqref{eq:sticky_fokker_planck_ab} has been carried out in \cite{vazquez2011heat} within a $H^1(\Omega)\times H^1(\G)$ functional framework.
For the smooth, positive solutions constructed therein, integration by parts gives
\begin{multline}
\label{eq:formal_dissipation}
\frac{d}{dt}\H(\rho_t\,\vert\,\mu)
=
\frac{d}{dt}\left(\int_\O \o_t\log\o_t +\int_\G\g_t\log\g_t\right)
\\
=
\int_\O(1+\log\o_t)\Delta\o_t +\int_\G(1+\log\g_t)\left[\Delta_\G\g_t -\partial_\nu\o_t\right]
\\
=
\left[-\int_\O|\nabla\log\o_t|^2\o_t + \int_{\pO}(1+\log\o_t)\partial_\nu\o_t\right]
\\
\qquad+
\left[-\int_\G|\nabla_\G\log\g_t|^2\g_t
-
\int_\G (1+\log\g_t)\partial_\nu\o_t\right]
\\
=
-\int_\O |\nabla\log\o_t|^2\o_t - \int_\G|\nabla_\G\log\g_t|^2\g_t
\eqqcolon -\I(\o_t)-\I(\g_t).
\end{multline}
In the last equality we used the Dirichlet boundary condition $\left.\o_t\right|_{\pO}=\g_t$ to cancel out the boundary terms (more on this in a moment).
This shows that the relative entropy $\H(\rho\,\vert\,\mu)=\int_\O\o\log\o+\int_\G\g\log\g$ of $\rho=\o+\g$ is dissipated by the sum of the full Fisher information $\I(\o)$ inside $\O$ and the tangential Fisher information $\I(\g)$ along the boundary $\G=\pO$.
Here our focus is not really on the well-posedness (although our analysis will provide as a byproduct existence of weak solutions for initial data with merely finite entropy), and we rather take interest in the variational gradient flow structure underlying this dissipation relation.

Although one can certainly come up with models naturally taking into account bulk/interface interactions directly into the transportation distance \cite{monsaingeon2021new,glitzky2013gradient}, we find surprising that the coupled system \eqref{eq:sticky_fokker_planck_GF} still fits within a completely standard optimal transport framework, as far as the metric is concerned.
Another striking aspect of our analysis is the following:
For standard Fokker-Planck equations the usual no-flux boundary condition encodes local conservation of mass, while in \eqref{eq:sticky_fokker_planck_GF} one may wonder where the boundary condition $\o|_\pO=\g$ stems from (this trace compatibility was crucial in order to get to \eqref{eq:formal_dissipation}).
This is actually unrelated to mass conservation:
Our two evolution equations for $\o,\g$ put together are already mass conservative regardless of the extra boundary condition $\o|_\pO=\g$, since $\partial_\nu\o$ in the second PDE is precisely the outflux of the momentum $\nabla\o$ appearing in the first one so whatever ``comes out of $\O$'' is just transferred to the boundary (the two PDEs together can indeed be reinterpreted as a single continuity equation $\partial_t\rho+\dive m=0$ with no-flux condition for the total density $\rho=\o+\g$, but with a singular part $\g$ living on the boundary in addition to a usual absolutely continuous part $\o$ in the interior).
The Dirichlet boundary condition $\o|_{\pO}=\g$ is also not merely encoded as a constraint in the entropy functional itself, since finiteness of $\H(\rho\,\vert\,\mu)$ certainly does not require or guarantee the trace compatibility.
For the sake of completeness let us mention that in \cite{figalli2010new} a transportation distance was constructed on $\M^+(\O)$, using the boundary $\G=\pO$ as an infinite reservoir allowing to store and release arbitrary amounts of mass.
It was then showed that the gradient flow of the standard entropy (with $\rd x$ as a reference measure) with respect to this new distance corresponds to the heat equation with Dirichlet boundary condition $\rho|_{\pO}=1$.
Although similar in spirit this is actually unrelated to our approach here, since our model is really mass conservative in all regards.
As we shall see later on, our boundary condition rather arises from the energy-dissipation mechanism, and more precisely from the \emph{metric slope} of the driving functional (see Theorem~\ref{theo:slope_controls_Fisher}).
To the best of our knowledge this is the first example of an optimal transport model in which entropy dissipation (rather than entropy itself) gives rise to Dirichlet boundary conditions.
In order to illustrate this even further, one could consider different energies, for example of the form
$$
\F(\rho)=\int_\O F_\O(\o)\rd x + \int_\G F_\G(\g)\rd x,\hspace{1cm}\rho=\o+\g
$$
for some functions $F_\O,F_\G:\R^+\to \R$ satisfying suitable structural conditions.
The gradient flow would then read
$$
\frac{d\rho}{dt}=-\grad_{\W}\F(\rho)
\qquad\Leftrightarrow\qquad
\begin{cases}
 \partial_t\o=\dive\left(\o\nabla F_\O'(\o)\right) & \mbox{in }\Omega\\
 \left.F'_\O(\o)\right|_{\pO}=F'_\G(\g) & \mbox{on }\partial\Omega\\
 \partial_t\g=\dive_\G\left(\g\nabla_\G F'_\G(\g)\right)-\o\partial_\nu F_\O'(\o) & \mbox{in }\G.
\end{cases}
$$
The scalar fields $F_\O'(\o),F'_\G(\g)$ are known in classical optimal transport as \emph{pressure variables} (in this respect the two evolution equations are nothing but Darcy's law).
Very heuristically, a pressure difference along $\G$ would create somehow an infinite force.
From a variational standpoint this should  be prohibited for dissipative systems.
Thus the Dirichlet boundary condition can be reinterpreted as a dissipative, pressure matching condition.
\\

On a slightly different note, the theory of gradient flows in abstract metric spaces and \emph{curves of maximal slope} was initiated by De Giorgi \cite{de1980problems}, and more recently developed in \cite{AGS}.
One possible way of formalizing the notion of such abstract gradient flows is to prove convergence of De Giorgi's \emph{minimizing movement} in the small time-step limit $\tau\to 0$, of which the original JKO scheme \cite{JKO98} is a particular instanciation in the Wasserstein space $(\P(\O),\W)$.
In order to back-up our claim \eqref{eq:sticky_fokker_planck_GF} that sticky diffusion is indeed a gradient flow we will pursue this by now classical approach, with however a few twists.
First, there are two classical ways of proving that the limiting curve is a ``solution''.
The first one consists in exploiting purely metric tools to retrieve in the limit an \emph{Energy Dissipation Inequality} (EDI in short), where a dissipation functional $\D(\rho_t)=-\frac d{dt}\E(\rho_t)$ plays a key role together with the \emph{metric speed} $|\dot\rho_t|^2$ (both computed with respect to the Wasserstein distance).
Usually the dissipation is related to the \emph{metric slope} $\D(\rho)=|\partial\E|^2(\rho)$, typically a Fisher information functional $\I(\rho)$, and this guarantees that the limit is a curve of maximal slope.
Moreover an upper chain rule argument shows that the energy dissipation forces equality in a Cauchy-Schwarz inequality, thus relating the driving momentum $m$ in the continuity equation $\partial_t\rho+\dive m=0$ to spatial gradients of the energy $m=-\rho\nabla\frac{\delta\E}{\delta\rho}$.
This gives the PDE in the end, but really requires a chain rule (computing the derivative in time of the energy $\E$ along a curve).
This is well understood for classical functionals \cite[chapter 10]{AGS}, less so here, and we did not fully succeed in this respect (see Theorem~\ref{theo:slope_controls_Fisher} and our Conjecture~\ref{conj}).

The second classical approach is directly PDE-oriented and less related to dissipation:
Leveraging ad-hoc tools from optimal transportation theory, one usually writes down the discrete Euler-Lagrange optimality condition for each step $\rho^n\leadsto\rho^{n+1}$ of the JKO scheme, and tries passing to the limit $\tau\to 0$ to retrieve the continuous PDE directly from the one-step optimality.
In the classical Fokker-Planck case this strongly relies on the Brenier-McCann theorem \cite{brenier1987decomposition,mccann2001polar}, allowing to relate the optimal map in $\rho^{n+1}=T^n\pf\rho^n$ with the energy.
More precisely, the Euler-Lagrange equation $\frac{T^n(x)-x}{\tau}=-\nabla\frac{\delta\E}{\delta\rho}(\rho^{n+1})$ typically gives the discrete velocity field driving particles around in the end (displacement $T^n(x)-x$ divided by time $\tau$).
This technical tool is unfortunately not available here, since our entropy functional always forces $\rho^n=\o^n+\g^n$ to have a singular part $\g^n$ supported on the boundary and therefore systematically prevents any application of the Brenier-McCann theorem (which precisely requires $\rho$ not to give mass to small $\H^{d-1}$ sets in dimension $d$!).
In some sense the bulk-interface interaction, corresponding to the $\partial_\nu\o$ exchange term in \eqref{eq:sticky_fokker_planck_GF}, forces mass splitting already at the discrete level.
In order to circumvent this technical obstacle and still obtain a useful Euler-Lagrange equation we perform instead an ad-hoc $\eps$-regularization of the entropy functional, and show that the barycentric momentum conveys enough information in order to first pass to the limit $\eps\to 0$ and then take $\tau\to 0$.
\\

Let us now fix some notations.
We always work in a smooth, bounded, convex domain $\O\subset\R^d$, $d\geq 2$.
(We assume convexity for technical convenience only, but the result should hold for general domains.)
We set
$$
H(z)
\coloneqq
z\log z -z+1 \geq 0,
\hspace{1cm}z\geq 0,
$$
and the reference measure will always be
$$
\mu
\coloneqq
\LO+\LG\in \M^+(\bO).
$$
The entropy
\begin{equation}
 \label{eq:def_E}
 \E(\rho)
 \coloneqq
 \H(\rho\,\vert\,\mu)
 =
 \begin{cases}
\int_\bO H\left(\frac{\rd \rho}{\rd\mu}(x)\right) \mu(\rd x) & \mbox{if }\rho\ll\mu\\
+\infty & \mbox{else}
 \end{cases}
\end{equation}
is then nonnegative, strictly convex, and lower semi-continuous.
As suggested by \eqref{eq:formal_dissipation} we define the dissipation
\begin{equation}
\label{eq:def_dissipation_D}
\D(\rho)\coloneqq
  \begin{cases}
\I(\o)+\I(\g) & \mbox{if }\rho=\o+\g\mbox{ and }\o|_{\pO}=\g\\
+\infty & \mbox{else}
  \end{cases}
\end{equation}
(see later on for a rigorous definition of the Fisher information functional $\I$ as well as the precise meaning of the boundary trace $\o|_{\pO}$).
This penalization of the constraint $\o|_\pO=\g$ is not technically artificial and arises intrinsically when trying to compute the metric slope $|\partial\E|$, see Theorem~\ref{theo:slope_controls_Fisher} below.

Given an initial datum $\rho_0\in \P(\bO)$, the JKO scheme consists in initializing $\rho^0=\rho_0$ and solving recursively
\begin{equation}
\label{eq:JKO}
 \rho^{n+1}
 \in
 \underset{\rho\in\P(\bO)}{\argmin}\left\{
 \frac{1}{2\tau}\W^2(\rho,\rho^n)+\E(\rho)
 \right\}.
\end{equation}
With such a discrete sequence one can define the piecewise constant interpolation
$$
\orho^\tau_t\coloneqq \rho^{n+1}
\qqtext{for }t\in (n\tau,(n+1)\tau],
$$
and we will establish
\begin{main_result}
Fix any $\rho_0\in \P(\bO)$ with $\E(\rho_0)<\infty$.
For any small $\tau>0$ the JKO scheme \eqref{eq:JKO} is well-posed, and there is a discrete sequence $\tau\to 0$ such that the interpolant $\orho^\tau$ converges to a continuous curve $\rho:[0,\infty)\to \P(\bO)$
$$
\W\left(\orho^\tau_t,\rho_t\right)\xrightarrow[\tau\to 0]{}0
\hspace{1cm}\mbox{for all }t\geq 0.
$$
The limit satisfies the Energy Dissipation Inequality
$$
\E(\rho_T)+\int_0^T\left(\frac 12\left|\dot \rho_t\right|^2 + \frac 12\D(\rho_t)\right)\rd t \leq \E(\rho_0),
 \hspace{1cm}\forall\,T>0
$$
and is a weak solution of the PDE \eqref{eq:sticky_fokker_planck_GF}.
Moreover, there is $\lambda>0$ only depending on $\O$ such that
\begin{equation}
\label{eq:long_time}
 \H(\rho_t\,\vert\,\bar\mu)\leq e^{-2\lambda t} \H(\rho_0\,\vert\,\bar\mu)
\qtext{and}
|\rho_t-\bar\mu|_{TV}\leq e^{-\lambda t}\sqrt{\frac 12 \H(\rho_0\,\vert\,\bar\mu)},
\qquad\forall\,t\geq 0
\end{equation}
where $\bar\mu\coloneqq \frac 1{\mu(\bO)}\mu$ it the renormalized stationary measure.
Finally, if $\rho_0\leq \overline c \mu$ for some constant $\overline c>0$ (resp. $\rho_0\geq \underline c\mu$ for some $\underline c$) then $\rho_t\leq \overline c \mu$ for all $t\geq 0$ (resp. $\rho_t\geq \underline c\mu$.)
\end{main_result}
It is worth stressing that something is still missing in order to obtain a rigorous metric gradient flow.
Indeed, in order for EDI to fully characterize curves of maximal slope one should really prove that $g(\rho)=\sqrt{\D(\rho)}$ is an \emph{upper gradient} \cite[chapter 1]{AGS}.
The local slope $|\partial\E|$ is always a (weak) upper gradient, and we conjecture that $\D=|\partial\E|^2$ thus it is plausible that $\sqrt{\D}$ should be an upper gradient.
However we only managed to prove that $\D\leq |\partial\E|^2$, and we were also not able to prove directly that $\sqrt{\D}$ is an upper gradient.

Another key concept for the abstract theory of metric gradient flows is that of geodesic convexity, or, in the specific optimal transport framework, McCann's displacement convexity \cite{mccann1997convexity}.
In smooth, complete Riemannian manifolds the $\lambda$-convexity of $\H$ (relatively to the reference volume measure) is equivalent to $\lambda$-Ricci lower bounds.
In Euclidean domains, it is known to hold with $\lambda=0$ if and only if $\Omega$ is convex, which can also be reinterpreted as the fact that the subspace $\P^{ac}_{\LO}(\bO)\subset\P(\bO)$ of absolutely continuous measures (w.r.t. to the reference Lebesgue measure $\LO$) is geodesically convex.
It is worth stressing that this whole picture collapses here:
Due to our choice of reference measure $\mu=\LO+\LG$ the entropy $\E(\rho)=\H(\rho\,\vert\,\mu)$ cannot be $\lambda$-convex for any $\lambda\in\R$, even if the underlying domain $\O$ is strongly uniformly convex.
This is due to the fact that $\P^{ac}_{\mu}(\bO)$ itself fails to be geodesically convex.
For a counterexample, take $\O$ the unit ball, and consider two measures $\rho_0=0+\g_0,\rho_1=0+\g_1$ smoothly supported on the boundary as in Figure~\ref{fig:non_convex}, with finite entropies $\E(\rho_i)=\H(\g_i\,\vert\,\LG)<\infty$.
\begin{figure}[!h]
\centering
\def\svgwidth{8cm}
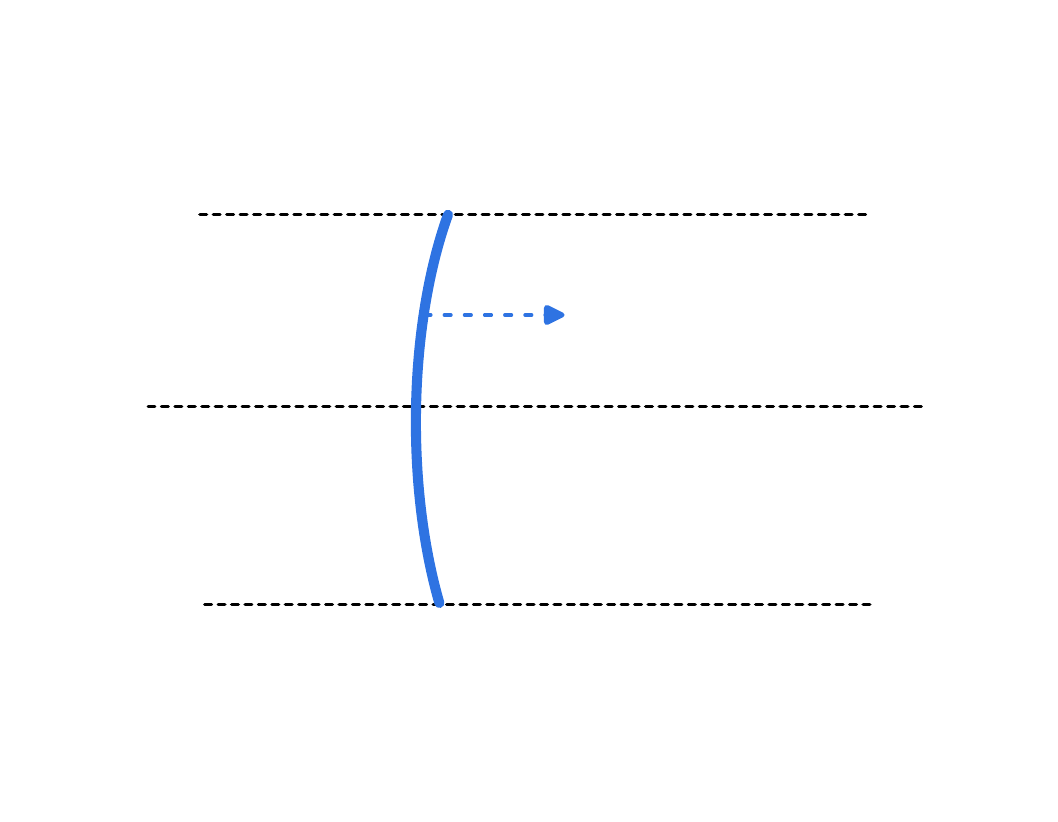
\caption{Counterexample to displacement convexity}
\label{fig:non_convex}
\end{figure}
It is not difficult in this configuration to check that the unique geodesic $\rho_t$ remains singular (and absolutely continuous w.r.t. $\H^{d-1}$) and supported strictly inside $\O$ for all intermediate times $t\in(0,1)$.
As a consequence $\rho_t=\o_t+0$ is not absolutely continuous w.r.t. $\mu$ and thus $\E(\rho_t)=+\infty$ for all $t\in(0,1)$, while $\E(\rho_0),\E(\rho_1)<+\infty$.

It is also known that geodesic convexity usually leads to logarithmic Sobolev and entropy-entropy production inequalities, which in turn should yield exponential convergence as $t\to\infty$ towards the unique entropy minimizer, here $\overline\mu=\frac{1}{|\O|+|\G|}\mu\in \P(\bO)$.
Although convexity completely fails here as just discussed, it was proved nonetheless in \cite{bormann2023functional} that $\overline\mu$ always satisfies a \emph{boundary-trace} logarithmic Sobolev inequality.
The usual entropy-entropy production argument will thus apply in our framework as well and will allow us to establish the exponential convergence \eqref{eq:long_time}.
\\

The paper is organized as follows:
In Section~\ref{sec:preliminaries} we fix some notations and define a few relevant objects.
We also establish a (partial) connection between the abstract metric slope $|\partial \E|(\rho)$ of our relative entropy $\E(\rho)=\H(\rho\,\vert\,\mu)$ and its dissipation functional $\D(\rho)$, Theorem~\ref{theo:slope_controls_Fisher}.
Section~\ref{sec:discrete_estimates} gathers several a-priori estimates for the discrete JKO solutions, including the Euler-Lagrange optimality condition and a discrete dissipation inequality based on De Giorgi's \emph{variational interpolant}.
In Section~\ref{sec:convergence} we leverage the previous work to pass to the limit $\tau\to 0$, and prove the main result.
For convenience we split this into Proposition~\ref{prop:abstract_convergence_curve} (convergence to a limiting curve satisfying EDI) and Proposition~\ref{prop:rho_weak_sol} (the limit is a weak solution of the PDE).
Finally, we defer to Appendix~\ref{sec:appendix} some technical statements and lemmas needed along the way.
%
\section{Preliminaries}
\label{sec:preliminaries}
Let us first fix once and for all the framework as well as a few notations.
\begin{itemize}
 \item
 $\O\subset\R^d$ ($d\geq 2$) is a smooth, bounded, convex domain with boundary $\G=\pO$.
 We tend to write $\pO$ for quantities arising from the interior, such as $\o|_{\pO}$, while we prefer the notation $\G$ for quantities intrinsically defined thereon, such as $\g$.
 The Euclidean distance on $\bO$ is denoted by $\dist_{\bO}$, while the intrinsic Riemannian distance on the boundary is denoted $\dist_\G$.
 The full gradient, divergence and Laplacian in $\Omega$ are denoted $\nabla,\dive,\Delta$, while their intrinsic Riemannian counterpart along $\G$ are denoted $\nabla_\G,\dive_\G,\Delta_\G$.
 \item
 We write $\X$ for smooth Riemannian manifolds, possibly with boundaries, and $\dist_\X$ the Riemannian distance.
 Typically we shall take $\X=\bO$, a flat manifold with boundary $\pO$, or $\X=\G$, a curved manifold of its own right (with induced Riemannian metric).
\item
 With a clear abuse of notations we call Lebesgue measure the volume form on a manifold $\X$ and we write indistinctly $\rd x$ for the Lebesgue measures $\LO$ or $\LG$. Unless otherwise specified, and with a slight abuse of notations, the Boltzmann entropy $\H(\nu)$ always denotes the entropy $\H(\nu\,\vert\,\rd x)$ of a nonnegative measure $\nu\in \M^+(\X)$ with respect to this $\rd x$ measure.
 \item
 For probability measures $\rho\in \P(\bO)$ we always decompose $\rho=\rho\mrest\O+\rho\mrest\G\eqqcolon \o+\g$ and identify $\o\in \M^+(\O),\g\in \M^+(\G)$ with their densities with respect to the corresponding Lebesgue measures, i-e we write $\o=\o(x)\cdot \LO$ and $\g=\g(x)\cdot \LG$.
 Typically this notation allows to simply write
 $$
 \E(\rho)
 =
 \H\left(\o+\g\,\vert\,\LO+\LG\right)
 =
 \H(\o)+\H(\g).
 $$
 \item
 The narrow convergence of measures is defined in duality with bounded, continuous test functions
 $$
 \rho^n\narrowcv\rho
 \qquad \mbox{iff}\qquad
 \int_\X \varphi(x)\rho^n(\rd x)  \to  \int_\X \varphi(x)\rho(\rd x),
 \quad\forall\,\varphi\in C_b(\X)
 $$
 as $n\to\infty$
 \item
 For $\mu,\nu\in \M^+(\X)$ with compatible masses $M=\int\mu=\int\nu$ the (squared) Wasserstein distance is
 $$
 \W^2_\X(\mu,\nu)=\min\limits_{\pi\in \Pi(\mu,\nu)}\iint_{\X\times\X}\dist_\X^2(x,y)\pi(\rd x,\rd y),
 $$
 where $\Pi(\mu,\nu)$ the set of admissible plans $\pi\in\M^+(\X\times\X)$ with mass $M$ and marginals $\pi_x=\mu,\pi_y=\nu$.
 We refer to \cite{OTAM,villani2003topics,villani_BIG} for more material on optimal transport theory.
 Unless otherwise specified $\W$ will simply denote $\W_{\bO}$, the Wasserstein distance over $\bO$ for the quadratic cost $\dist_\bO^2(x,y)=|y-x|^2$.
 \item
 Following \cite{AGS}, in a metric space $(\mathsf X,\dist)$ a curve $\{x_t\}_{t\in I}$ is said to be absolutely continuous if
 $$
 \dist(x_{t_0},x_{t_1})\leq \int_{t_0}^{t_1} \eta(t)\rd t
 $$
 for some $\eta\in L^1(I)$ and all $[t_0,t_1]\subseteq I$.
 In that case the \emph{metric speed}
 $$
 |\dot x_t|\coloneqq \lim\limits_{h\to 0}\frac{\dist(x_{t+h},x_t)}{h}
 $$
 exists for a.e. $t\in I$ and is the smallest function $\eta$ satisfying the above inequality.
 The (local) \emph{metric slope} of a functional $\mathsf E:\mathsf X\to \R\cup\{+\infty\}$ is defined as
 \begin{equation}
\label{eq:def_metric_slope}
 |\partial \mathsf E|(x)=\limsup\limits_{y\to x}\frac{[\mathsf E(x)-\mathsf E(y)]^+}{\dist(x,y)}.
 \end{equation}
 \item
 We use the $t$ subscript for curves of measures $t\mapsto\rho_t=\o_t+\g_t$, while the time derivative is denoted $\partial_t$.
\end{itemize}

\begin{defi}[Fisher information {\cite[defs. 10.4.15 and 10.4.16]{AGS}}]
On a smooth Riemannian manifold (possibly with boundary) $\X$ we say that a nonnegative measure $\mu\in \M^+(\X)$ has logarithmic derivative if $\mu=f(x)\cdot \rd x$ for some $f\in L^1(\rd x)$ and if its distributional gradient $\nabla f\in L^1_{loc}(\rd x)$ with moreover $|\nabla f|\ll f$.
In that case we write with a slight abuse of notations $\nabla\log f=\frac{\nabla f}{f}$ for the Radon-Nikodym derivative.
The Fisher information is then
$$
\I(\mu)\coloneqq
\begin{cases}
 \int_\X |\nabla\log f|^2 \rd \mu & \mbox{if }\mu=f\cdot\rd x\mbox{ has logarithmic derivative }\nabla\log f\\
 +\infty & \mbox{otherwise}
\end{cases}
$$
\end{defi}
We shall often exploit the following well-known facts, which we recall here without proof.
\begin{lem}[Properties of the Fisher information]
\label{lem:props_Fisher}
 For any nonnegative, absolutely continuous $f\in L^1(\X)$ with mass $M=\int_\X f\,\rd x<\infty$ we write $\mu=f\cdot \rd x$.
 Then
 \begin{itemize}
 \item
 One has the equivalent representations
 \begin{equation*}
 \I(\mu)
 =
 \int_\X \left|\nabla\log f\right|^2f
 =
 \int_\X \left|\frac{\nabla f}{f}\right|^2f
 =
 \int_\X\left|\frac{\nabla f}{\sqrt f}\right|^2
 =
 4\left\|\nabla\sqrt f\right\|^2_{L^2}
  \end{equation*}
 \item
 In dimension $d$ there holds
 \begin{equation}
 \label{eq:Fisher_mass_control_W1p}
 \left\|f\right\|_{W^{1,p}(\X)}\leq C_{p,\X}[M+\I(\mu)]
 \qqtext{for}
 \begin{cases}
p=2 & \mbox{if } d=1
\\
p=2^- & \mbox{if }d=2
\\
p=\frac{d}{d-1} & \mbox{if }d\geq 3
 \end{cases}.
 \end{equation}
  ($p=2^-$ means as usual ``for any $p<2$'')
 \item
 $\I$ is convex and lower semi-continuous for the narrow convergence.
\end{itemize}
\end{lem}
Recall that, on smooth bounded domains $\O\subset\R^d$, the trace operator $\tr:W^{1,p}(\O)\to L^p(\pO)$ is continuous for $p\geq 1$ and compact for $p> 1$ \cite[chapter 18]{leoni2017first}.
Owing to \eqref{eq:Fisher_mass_control_W1p} we see that a measure $\mu=f\cdot \rd x$ with $\I(\mu)<\infty$ always has a well-defined trace $f|_\pO=\tr f\in L^{p}(\pO)$ for some $p>1$.
This clarifies the rigorous meaning of the dissipation functional $\D$ in \eqref{eq:def_dissipation_D}:
Either $\I(\o)$ and $\I(\g)$ are finite and therefore the trace $\o|_\pO=\tr \o$ is well-defined, or simply $\D(\rho)=+\infty$.
This will also give a meaning to the boundary condition $\o|_\pO=\g$ in the Fokker-Planck equation \eqref{eq:sticky_fokker_planck_GF}, whose solutions will satisfy by construction $\int_0^T\D(\rho_t)\,\rd t<+\infty$ and thus $\D(\rho_t)<+\infty$ for a.e. $t\geq 0$.
\\

The next result is specific to our optimal transport context with $\E(\rho)=\H(\o)+\H(\g)$, and will be key for the subsequent analysis.
It will allow to retrieve simultaneously the boundary conditions as well as natural dissipation estimates for the PDE directly from the small-time step limit $\tau\to 0$ in the JKO scheme, and more importantly justifies our claim in the introduction that the trace compatibility arises from the dissipation only.
\begin{theo}
\label{theo:slope_controls_Fisher}
Let $\D$ be the dissipation functional in \eqref{eq:def_dissipation_D}.
Then
 \begin{equation}
 \label{eq:slope_control}
  |\partial\E|^2(\rho)
  \geq
  \D(\rho)
 \end{equation}
 and $\D$ is (sequentially) narrowly lower semi-continuous.
\end{theo}
We were not able to prove the reverse inequality, and for future reference we record here
\begin{conj}
\label{conj}
Equality holds in \eqref{eq:slope_control}.
\end{conj}
\begin{proof}[Proof of Theorem~\ref{theo:slope_controls_Fisher}]
Let us first show that $|\partial\E|^2(\rho)\geq \I(\o)+\I(\g)$.
Take $\rho=\o+\g\in\Dom(\E)$ with masses $M_\O=\int_\O\o\geq 0$ and $M_\G=\int_\G\g\geq 0$, $M_\O+M_\G=1$.
For any arbitrary $\tilde\o\in \M^+(\bO),\tilde\g\in\M^+(\G)$ with same masses $M_\O,M_\G$ let $\tilde\rho=\tilde\o+\tilde\g\in\P(\bO)$.
Let $\pi_\O\in\Pi(\o,\tilde\o)$ and $\pi_\G\in \Pi(\g,\tilde\g)$ be optimal plans for the transportation problems set on $(\bO,\dist_\bO)$ and $(\G,\dist_\G)$, respectively.
The plan $\pi\coloneqq \pi_\O+\pi_\G$ is easily seen to be admissible for the transportation problem from $\rho=\o+\g$ to $\tilde\rho=\tilde\o+\tilde\g$, set on $(\bO,\dist_\bO)$.
Since the induced distance $\dist_\G(x,y)\geq \dist_\bO(x,y)$ along the boundary we have in particular
\begin{multline}
\W^2(\rho,\tilde\rho)
\leq \iint_{\bO^2} \dist^2_\bO(x,y)\pi(\rd x,\rd y)
=\iint_{\O^2} \dist^2_\bO(x,y)\pi_\O(\rd x,\rd y) + \iint_{\G^2} \dist^2_\bO(x,y)\pi_\G(\rd x,\rd y)
\\
\leq
\iint_{\O^2} \dist^2_\bO(x,y)\pi_\O(\rd x,\rd y) + \iint_{\G^2} \dist^2_\G(x,y)\pi_\G(\rd x,\rd y)
=\W^2_\bO(\o,\tilde\o) + \W^2_\G(\g,\tilde\g).
\label{eq:estimate_W_WO_WG}
\end{multline}
Let now $(\o_t,\g_t)_{t\geq 0}$ be the heat flows running on $\O,\G$, respectively, started from $\o,\g$  (with of course no-flux boundary condition on $\pO$ for the first flow).
Recall at this stage the \emph{duality formula for the local slope} \cite[lemma 3.1.5]{AGS}
$$
\frac 12|\partial\E|^2(\rho)=\limsup\limits_{t\downarrow 0}\frac{\E(\rho)-\E_t(\rho)}{t},
$$
where $\E_t(\rho)=\inf\limits_{\tilde\rho}\left\{\E(\tilde\rho)+\frac 1{2t}\W^2(\rho,\tilde\rho)\right\}$ is the Moreau-Yosida regularisation of $\E$.
Writing $\rho_t=\o_t+\g_t$ we can bound
\begin{multline*}
\frac 12|\partial\E|^2(\rho)
=
\limsup\limits_{ t\downarrow 0}\frac 1 t\sup\limits_{\tilde\rho}\left\{\E(\rho)-\E(\tilde\rho)-\frac 1{2 t}\W^2(\rho,\tilde\rho)\right\}
\\
=
\limsup\limits_{ t\downarrow 0}\frac 1 t\sup\limits_{\tilde\rho=\tilde\o+\tilde\g}\left\{\left[\H(\o)+\H(\g)\right]-\left[\H(\tilde\o)+\H(\tilde\g)\right]-\frac 1{2 t}\W^2(\rho,\tilde\rho)\right\}
\\
\geq
\limsup\limits_{t\downarrow 0}\frac 1t\left\{\left[\H(\o)+\H(\g)\right]-\left[\H(\o_t)+\H(\g_t)\right]-\frac 1{2t}\W^2(\rho,\rho_t)\right\}
\\
\overset{\eqref{eq:estimate_W_WO_WG}}{\geq }
\limsup\limits_{t\downarrow 0}\frac 1t\left\{\left[\H(\o)-\H(\o_t)-\frac 1{2t}\W^2_\O(\o,\o_t)\right]+\left[\H(\g)-\H(\g_t)-\frac 1{2t}\W^2_\G(\g,\g_t)\right]\right\}
\\
\geq
\liminf\limits_{t\downarrow 0}\frac 1t\left[\H(\o)-\H(\o_t)-\frac 1{2t}\W^2_\O(\o,\o_t)\right]
+\liminf\limits_{t\downarrow 0}\frac 1t\left[\H(\g)-\H(\g_t)-\frac 1{2t}\W^2_\G(\g,\g_t)\right].
\end{multline*}
By Proposition~\ref{prop:slope_controls_Fisher} the two $\liminf$ are bounded from below by half the respective Fisher informations, and $|\partial\E|^2(\rho)\geq \I(\o)+\I(\g)$ follows.

Let us now prove the the lower semicontinuity.
Take a sequence $\rho^n\narrowcv\rho$.
Up to extraction of a subsequence if needed we can assume that $\liminf \D(\rho^n)=\lim\D(\rho^n)<\infty$.
By definition of $\D$ we have $\I(\o^n)+\I(\g^n)\leq C$ and $\left.\o^n\right|_{\pO}=\g^n$ for all $n$, and from \eqref{eq:Fisher_mass_control_W1p} we have moreover that $\o^n,\g^n$ are bounded in $W^{1,p}$ for some $p>1$.
By Rellich-Kondrachov compactness we get that $\o^n\to \o$ and $\g^n\to \g$ at least in $L^1(\O)$, $L^1(\G)$, respectively.
This implies in particular that $\rho=\lim\rho^n$ has interior/boundary decomposition $\o+\g=\lim\o^n+\lim\g^n$.
 By compactness of the trace operator we also get $\o|_{\pO}=\tr \o=\lim (\tr \o^n)=\lim \g^n=\g$, hence $\D(\rho)=\I(\o)+\I(\g)$.
 By standard lower semi-continuity of $\I$ we have next $\D(\rho)=\I(\o)+\I(\g)\leq \liminf\I(\o^n)+\liminf\I(\g^n)\leq \liminf \{\I(\o^n)+\I(\g^n)\}=\liminf \D(\rho^n)$ and the claim follows.

It only remains now to address the much more delicate trace compatibility, i.e. we need to prove that if $\I(\o)+\I(\g)<\infty$ but $\o|_{\pO}\neq \g$ then $|\partial\E|(\rho)=+\infty$.
Note that this is meaningful: on the one hand since $\I(\o)<\infty$ Lemma~\ref{lem:props_Fisher} shows that $\o\in W^{1,p}(\O)$ for some $p>1$, hence the trace $\o|_{\pO}=\tr(\o)\in L^p(\pO)$ is well-defined, and on the other hand $\g\in L^1(\G)$ so both $\o|_{\pO}, \g$ are unambiguously defined $\LG$ a.e.
Assume that there is a common Lebesgue point $x_0\in\pO=\G$ such that $\o|_{\pO}(x_0)\neq \g(x_0)$.
Consider first the case
$$
\o_0\coloneqq\o|_{\pO}(x_0) > \g(x_0)\eqqcolon \g_0.
$$
For some small $\eps,\delta,\theta$ to be adjusted later on we will construct below a suitable $(\eps,\delta,\theta)$-perturbation $\tilde\o,\tilde\g$ of $\o,\g$ in a neighborhood of $x_0$ by transferring a small fraction $\theta$ of mass from the interior to the boundary in order to ``diminish the gap $\o_0-\g_0>0$'' close to $x_0$.
This will be achieved by paying a transportation cost one order of magnitude smaller that the gain in entropy as $\eps,\delta,\theta\to 0$, hence by definition \eqref{eq:def_metric_slope} of the slope we will have $|\partial \E|(\rho)\geq \liminf\limits_{\eps,\delta,\theta\to 0}\frac{\E(\rho)-\E(\tilde\rho)}{\W(\rho,\tilde\rho)}=+\infty$.
Roughly speaking $\delta$ will be a small tangential length-scale along the boundary and $\eps$ a small normal scale, so that the perturbation will be constructed in a box of size $\delta^{d-1}\eps$.
The parameter $\theta$ will control the small fraction of mass to be taken from the interior and mapped to the boundary, and in some sense $\theta$ will be used to ``linearise'' $H(\o)-H(\tilde\o)\sim H'(\tilde\o)(\o-\tilde\o)$ and $H(\g)-H(\tilde\g)\sim H'(\tilde\g)(\g-\tilde\g)$.
The important point is that the transfer of mass from the interior to the boundary should take place on a normal scale $\eps\ll\delta$ much smaller than the tangential scale $\delta$.
This reflects the normal derivative $\partial_\nu\o$ appearing in \eqref{eq:sticky_fokker_planck_GF}, which consistently encodes the bulk/interface interactions.

Since our construction of the perturbation below will be localized, it is enough to work in local charts and we thus work in the following configuration: $\O=\R^d_+$ is a half-space, $\G=\pO=\{0\}\times\R^{d-1}$, and $x_0=(0,0_{\R^{d-1}})\in \G$.
We write
\begin{equation*}
x=(x_1,x')\in \R^+\times \R^{d-1} \qqtext{and} T(x)=\operatorname{Proj}_{\R^{d-1}}(x)=(0,x')
\end{equation*}
for the natural coordinates and projection on the boundary.
We take $\G_\delta=B^{d-1}_\delta(0)$ a small neighborhood of the origin in $\G$, and also define the $\eps$-thickening $\Oed=(0,\eps)\times\G_\delta$ as in Figure~\ref{fig:eps_delta}.
For reasons that will appear clear later on we impose
\begin{equation}
\label{eq:regime_eps_delta}
\eps^{p-1}=\mathcal O\left(\delta^{d-1}\right)
\hspace{1cm}
\mbox{for }p=\frac{d}{d-1}>1.
\end{equation}
\begin{figure}
 \begin{center}
 \begin{tikzpicture}[xscale=.7]
  \draw  [->] (-1,0) to (6,0);
  \node [below] at (6,0) {$x_1\in \R^+$};
  \draw  [->] (0,-4) to (0,4);
  \node [left] at (0,4) {$x'\in \R^{d-1}$};

  \node [left] at (0,-.3) {$0$};
  \filldraw [fill=red, nearly transparent] (0,-3.5) rectangle (4,3.5);
  \node [red] at (3,-3) {$\Oed$};
  \draw [line width = 2pt, line cap=round, blue] (0,-3.5) to (0,3.5);
  \node [blue] at (-.5,-3) {$\Gd$};
  \filldraw (0,0) circle (2pt);
  \filldraw (3,2.5) circle (2pt);
  \node [right] at (3,2.5) {$x=(x_1,x')$};
  \draw [dashed] (3,2.5) to (0,2.5);
  \node [left] at (0,2.5) {$T(x)=(0,x')$};
  \filldraw (0,2.5) circle (2pt);
  \draw [<->] (4.3,-3.4) to (4.3,-.1);
  \node [right] at (4.3,-1.5) {$\delta$};
  \draw [<->] (0.1,-3.7) to (3.9,-3.7);
  \node [below] at (2,-3.7) {$\eps$};
 \end{tikzpicture}

 \end{center}
 \caption{The $(\eps,\delta)$ boundary layer}
\label{fig:eps_delta}
 \end{figure}
Writing $\hat \g\coloneqq T\pf(\o\mrest\Oed)$ for simplicity, the perturbation is then defined as
 \begin{equation}
  \label{eq:def_perturbation_rho_tilde}
 \tilde \o(x)\coloneqq
\begin{cases}
(1-\theta)\o(x) & \mbox{if }x\in \Oed\\
\o(x) & \mbox{else}
\end{cases}
\qqtext{and}
\tilde\g(x')
 \coloneqq
 \begin{cases}
 \g(x') + \theta\hat\g(x') & \mbox{if }x'\in \G_\delta\\
 \g(x') & \mbox{else}
 \end{cases}
 \end{equation}
In other words, we take a small fraction $\theta\ll 1$ of $\o$ inside $\Oed$ and transport it to the boundary $\G_\delta$ along the $x_1$ direction, where it is simply added to the boundary density already present there.
Meanwhile, $\o,\g$ remain unchanged outside of the small control boxes $\Oed,\Gd$, respectively.
Notice that $\tilde\o,\tilde\g$ do \emph{not} satisfy the compatibility condition.
We perturb however ``in the good direction'' by decreasing the trace gap, since $\tilde\o\leq \o$ and $\tilde \g\geq \g$ and therefore $[\tilde \o-\tilde\g](x_0)\leq [\o-\g](x_0)$.
\paragraph{Step 1: variation in entropy.}
In order to estimate the first variation of $\E$ we first exploit the convexity of $H$ to write
$$
\E(\rho)-\E(\trho)
=
[\H(\o)-\H(\tilde\o)]+[\H(\g)-\H(\tilde\g)]
\geq
\underbrace{\int_{\Oed}H'(\tilde\o)[\o-\tilde\o]}_{\coloneqq A}
+
\underbrace{\int_{\Gd}H'(\tilde\g)[\g-\tilde\g]}_{\coloneqq B}
$$
For the first term we use the convexity of $z\mapsto z H'(z)=z\log z$ to apply Jensen's inequality as
\begin{multline*}
A
=
\int_{\Oed}H'[(1-\theta)\o]\theta \o
=
\frac{\theta}{1-\theta}\int_{\Oed}H'[(1-\theta)\o](1-\theta)\o
\\
=
\frac{\theta|\Oed|}{1-\theta}\fint_{\Oed}H'[(1-\theta)\o](1-\theta)\o
\\
\geq
\frac{\theta|\Oed|}{1-\theta}H'\left((1-\theta)\fint_{\Oed}\o\right)\left((1-\theta)\fint_{\Oed}\o\right).
\end{multline*}
From Lemma~\ref{lem:props_Fisher} we have that $\o\in W^{1,p}$ with $p=\frac{d}{d-1}>1$.
Because we took care to impose the $\eps,\delta$ regime \eqref{eq:regime_eps_delta} we can apply Lemma~\ref{lem:W1p_boundary_lebesgue_diff} to conclude that $\fint_\Oed\o\to \o_0>0$ and thus
\begin{equation}
\label{eq:liminf_A}
\liminf\limits_{\eps,\delta,\theta}\frac{A}{\theta|\Oed|} \geq H'(\o_0)\o_0.
\end{equation}
(Note that, since we assumed $\o_0>\g_0$, we have in particular $\o_0>0$ so there is no issue with $\log 0$ here.)
The $B$ term is more involved and requires extra work.
Observe that $\hat \g=T\pf(\o\mrest\Oed)$ is given more explicitly by
$$
\hat \g(x')=\int_0^\eps \o(x_1,x')\rd x_1,
\hspace{1cm}x'\in \Gd.
$$
This is small as $\eps\to 0$ so in order to obtain something nontrivial one should actually look at the rescaled quantity
$$
\bar\g(x') \coloneqq \frac 1\eps \hat \g(x')
=
\int_0^\eps \o(x_1,x')\frac{\rd x_1}{\eps}.
$$
For any $p\geq 1$ we have by Jensen's inequality
\begin{multline}
\label{eq:gbar_to_o0}
\fint_\Gd |\bar \g -\o_0|^p
=
\fint_\Gd \left|\int_0^\eps [\o(x_1,x') -\o_0]\frac{\rd x_1}{\eps}\right|^p\rd x'
\\
\leq
\fint_\Gd \int_0^\eps |\o(x_1,x') -\o_0|^p\frac{\rd x_1}{\eps}\rd x'
=\fint_\Oed |\o-\o_0|^p\to 0
\end{multline}
as soon as $p\leq \frac{d}{d-1}$, due to Lemma~\ref{lem:props_Fisher} and Lemma~\ref{lem:W1p_boundary_lebesgue_diff} exactly as before.
This confirms that $\bar\g$ remains indeed of order $\o_0=\mathcal O(1)$ in any $L^p,p\leq \frac d{d-1}$.
Fixing any $M>\g_0$, we define now
$$
\ulogM(z)\coloneqq H'(z)\chi_{(0,M]}(z)
\qqtext{and}
\ologM(z)\coloneqq H'(z)\chi_{(M,\infty)}(z),
$$
the lower and upper truncations of $H'(z)=\log z$.
We split next, after changing variables $x'=\delta y'$ and writing $\phi^\delta(y')=\phi(\delta y')$ for any relevant function $\phi=\g,\tilde\g,\hat\g,\bar\g$,
\begin{multline*}
B=\int_{\Gd}H'(\tilde\g)[\g-\tilde\g]
=-|\Gd|\fint_\Gd H'(\g+\eps\theta\bar\g)\eps\theta\bar\g
=
-\theta\eps|\Gd|\fint_{B^{d-1}_1} H'(\g^\delta+\eps\theta\bar\g^\delta)\bar\g^\delta
\\
=\theta|\Oed|\Bigg( \underbrace{\fint_{B^{d-1}_1} -\ulogM(\g^\delta+\eps\theta\bar\g^\delta)\bar\g^\delta }_{\coloneqq \underline B}- \underbrace{\fint_{B^{d-1}_1} \ologM(\g^\delta+\eps\theta\bar\g^\delta)\bar\g^\delta}_{\coloneqq\overline B }\Bigg)
\end{multline*}
Let us first handle the $\underline B$ term.
By Lebesgue's differentiation theorem we have
$$
\fint_{B^{d-1}_1}|\g^\delta-\g_0|
=
\fint_{B^{d-1}_{\delta}} |\g-\g_0|\to 0,
$$
in particular $\g^\delta\to \g_0$ a.e. in $B^{d-1}_1$.
Owing to \eqref{eq:gbar_to_o0} we have similarly
$$
\fint_{B^{d-1}_1}|\bar\g^\delta-\o_0|
=
\fint_{B^{d-1}_{\delta}} |\bar\g-\o_0|
\leq
\fint_{\Oed} |\o-\o_0| \to 0,
$$
hence $\bar\g^\delta\to \o_0$ a.e as well.
As a consequence $-\ulogM(\g^\delta+\eps\theta\bar\g^\delta)\bar\g^\delta\to -\ulogM(\g_0)\o_0=-H'(\g_0)\o_0$ at least a.e. inside $B^{d-1}_1$ (here it is crucial that the truncation in $\ulogM$ was chosen at level $M>\gamma_0$ so that $\ulogM(\g_0)=H'(\g_0)$).
Since $H'(M)-\ulogM(z)\geq 0$ for fixed $M$ (this is mainly why we had to truncate) we conclude by Fatou's lemma that
\begin{multline*}
\liminf \limits_{\eps,\delta,\theta} \underline B
=
\liminf \limits_{\eps,\delta,\theta}\left\{ \fint_{B^{d-1}_1} \left[H'(M)-\ulogM(\g^\delta+\eps\theta\bar\g^\delta)\right]\bar\g^\delta + \fint_{B^{d-1}_1}-H'(M)\bar\g^\delta
\right\}
\\
=
\liminf \limits_{\eps,\delta,\theta}\left\{\fint_{B^{d-1}_1} \left[H'(M)-\ulogM(\g^\delta+\eps\theta\bar\g^\delta)\right]\bar\g^\delta\right\}
-H'(M)\lim \limits_{\eps,\delta,\theta}\fint_{B^{d-1}_1} \bar\g^\delta
\\
\geq
\fint_{B^{d-1}_1}[H'(M)-H'(\g_0)]\o_0 -H'(M)\o_0
=-H'( \g_0)\o_0.
\end{multline*}
Note that this allows for $\g_0=0$ with $-(\log \g_0)\o_0=+\infty$, since in that case $\o_0>\g_0=0$.

For the $\overline B$ term we write
$$
-\overline B=\fint_{B^{d-1}_1}
\underbrace{\ologM(\g^\delta+\eps\theta\bar\g^\delta)}_{\coloneqq f}
\underbrace{\bar\g^\delta}_{\coloneqq g}.
$$
Exactly as before, and again because we truncated at level $M>\g_0$, we have convergence $\ologM(\g^\delta+\eps\theta\bar\g^\delta)\bar\g^\delta\to \ologM(\g_0+0)\o_0=0$ a.e.
In order to apply Vitali's convergence theorem it suffices to show that the product $\{fg\}_{\eps,\delta,\theta}$ is uniformly integrable in $B^{d-1}_1$ as $\eps,\delta,\theta\to 0$.
Since $\ologM(z)=(\log z)\chi_{(M,\infty)}(z)$ has been truncated one has that the negative part $[\ologM]^-(z)\leq \max\{0,-\log M\}$, which in turn easily yields
$$
\exp\left(\left|\ologM(z)\right|\right)\leq C_M(1+z),
\hspace{1cm}
\forall\,z\geq 0
$$
for some suitable $C_M>0$.
Whence
$$
\fint_{B^{d-1}_1}\exp f
\leq
\fint_{B^{d-1}_1}C_M\left(1+ \g^\delta+\eps\theta\bar\g^\delta)\right)
=
C_M\fint_{B^{d-1}_\delta}(1+ \g+\eps\theta\bar\g)\to C_M\left[1+\g_0+0\right].
$$
Moreover by Lemma~\ref{lem:props_Fisher} we have $\o\in W^{1,\frac d{d-1}}$.
One more application of Lemma~\ref{lem:W1p_boundary_lebesgue_diff} gives
$$
\fint_{B^{d-1}_1}|g|^\frac{d}{d-1}
=
\fint_{B^{d-1}_\delta}|\bar\g|^\frac{d}{d-1}\to \o_0^\frac{d}{d-1}.
$$
We use now the elementary inequality $(ab)^p \leq C\left[1+\exp(a)+b^q\right]$ for $a,b\geq 0$ and any $1\leq p< q$ (for some $C=C_{p,r}$).
Indeed, choosing any $p\in \left(1,\frac{d}{d-1}\right)$ we conclude from the above estimates that
$$
\fint_{B^{d-1}_1}|fg|^p
\leq
\fint_{B^{d-1}_1} C\left(1+\exp(f) + |g|^\frac{d}{d-1}\right)
=
\mathcal O(1)
$$
as $\eps,\delta,\theta\to 0$, for some fixed $p>1$.
This gives equi-integrability and we can thus legitimately apply Vitali's convergence theorem to conclude that $\lim\limits_{\eps,\delta,\theta\to 0}\overline B =0$.
Whence
\begin{equation}
 \label{eq:liminf_B}
 \liminf\limits_{\eps,\delta,\theta\to 0} \frac{B}{\theta|\Oed|}
 \geq
 \liminf\limits_{\eps,\delta,\theta\to 0} \frac{\underline B}{\theta|\Oed|}
 \geq -H'(\g_0)\o_0.
\end{equation}
Gathering \eqref{eq:liminf_A}\eqref{eq:liminf_B}, and by strict convexity of $H$ with $\o_0>\g_0$, we obtain
\begin{equation}
\label{eq:variation_entropy}
 \liminf\limits_{\eps,\delta,\theta\to 0} \frac{\E(\rho)-\E(\trho)}{\theta|\Oed|}
 \geq [H'(\o_0)-H'(\g_0)]\o_0 >0.
\end{equation}
\paragraph{Step 2: transportation cost.}
Recalling the definition \eqref{eq:def_perturbation_rho_tilde} of our perturbation $\trho=\tilde\o+\tilde\g$, it is easy to check that the plan $\pi\in \P\left(\bO\times\bO\right)$ defined by duality as
$$
\iint_{\bO\times\bO} \varphi(x,y)\pi(\rd x,\rd y)
=
\int\limits_{\bO\setminus\Oed}\varphi(x,x)\rho(\rd x)
+ (1-\theta) \int\limits_{\Oed}\varphi(x,x)\rho(\rd x)
+ \theta \int\limits_{\Oed}\varphi(x,T(x))\rho(\rd x)
$$
is admissible from $\rho$ to $\trho$.
Since $|T(x)-x|\leq \eps$ in $\Oed$ we obtain
\begin{multline}
\label{eq:W2_leq_eps}
\W^2(\rho,\trho)
\leq
\iint_{\bO^2}|x-y|^2 \pi(\rd x,\rd y)
= \theta \int_{\Oed}|x-T(x)|^2\rho(\rd x)
\\
= \theta \int_{\Oed}|x-T(x)|^2\o(\rd x)
\leq
\theta |\Oed|\eps^2\fint_\Oed \o
\sim
\theta |\Oed|\eps^2\o_0.
\end{multline}
\paragraph{Step 3: conclusion.}
Take $\eps$ small enough so that \eqref{eq:regime_eps_delta} and therefore \eqref{eq:variation_entropy}\eqref{eq:W2_leq_eps} hold.
Recalling that $|\Oed|=|B^{d-1}_1|\eps\delta^{d-1}$ we obtain at last
\begin{multline*}
|\partial\E(\rho)|=\limsup\limits_{\mu\to\rho}\frac{\left[\E(\rho)-\E(\mu)\right]^+}{\W(\rho,\mu)}
\geq
\liminf\limits_{\eps,\delta,\theta\to 0}\frac{\theta|\Oed|[H'(\o_0)-H'(\g_0)]\o_0}{\left(\theta |\Oed|\eps^2\o_0\right)^{\frac 12}}
\\
=
\underbrace{\sqrt{\o_0}[H'(\o_0)-H'(\g_0)]}_{=C_0>0}
\liminf\limits_{\eps,\delta,\theta\to 0}\frac{\sqrt{\theta |\Oed|}}{\eps}
=
C_0'\liminf\limits_{\eps,\delta,\theta\to 0}\left(\frac{\theta\delta^{d-1}}{\eps}\right)^\frac 12
=+\infty
\end{multline*}
provided $\eps\ll \theta \delta^{d-1}$ - which is of course compatible with \eqref{eq:regime_eps_delta}.
This finally settles the case $\o_0>\g_0\geq 0$.

If now $\g_0>\o_0\geq 0$ we proceed similarly, by spreading uniformly along horizontal lines the excess of $\g$-mass from the boundary $\Gd$ to the interior $\Oed$.
More explicitly the new perturbation is defined as
$$
\tilde \g(x')\coloneqq
\begin{cases}
(1-\theta)\g(x') & \mbox{if }x'\in \Gd\\
\g(x') & \mbox{else}
\end{cases}
\qqtext{and}
\tilde\o(x)
 \coloneqq
 \begin{cases}
 \o(x) + \theta\frac{1}{\eps}\g(x') & \mbox{if }x=(x_1,x')\in \Oed\\
 \g(x) & \mbox{else}
 \end{cases}.
$$
The rest of the argument is very similar and for the sake of brevity we omit the details.
(This case is actually easier because the horizontal oscillations of the infinitesimal variations $\frac{1}{\eps}\g(x')$ of $\o$ are trivially under control, while they previously required some $W^{1,p}$ control via the Fisher information in order to apply Lemma~\ref{lem:W1p_boundary_lebesgue_diff}).
\end{proof}

\section{Discrete estimates}
\label{sec:discrete_estimates}
In this section we lay the groundwork for the convergence $\rho^\tau_t\to\rho_t$ later on.
We start with
\begin{lem}
\label{lem:exists_minimizer_JKO}
 Given any $\rho^n\in \P(\bO)$ there exists a unique minimizer $\rho^{n+1}=\o^{n+1}+\g^{n+1}$ of \eqref{eq:JKO}, satisfying
 $$
 |\partial \E|(\rho^{n+1})\leq \frac{\W(\rho^n,\rho^{n+1})}{\tau}<+\infty
 $$
 with moreover $\o^{n+1}\in W^{1,\frac{d}{d-1}}(\O),\g^{n+1}\in W^{1,\frac{d-1}{d-2}}(\G)$ and the trace compatibility
 $$
\left.\o^{n+1}\right|_{\pO}=\g^{n+1}.
$$
\end{lem}

\begin{proof}
Recall first that $\E$ and $\W^2(\cdot,\rho^n)$ are classically nonnegative and lower semi-continuous w.r.t. narrow convergence.
The latter is linearly convex, while the former is strictly convex.
 Hence the direct method in the calculus of variations gives existence and uniqueness of a minimizer $\rho^{n+1}$ for the JKO functional $ \F^n_\tau(\rho)\coloneqq \frac 1{2\tau}\W^2(\rho,\rho^{n})+\E(\rho)$ in \eqref{eq:JKO}.
 Moreover this minimizer has finite entropy and therefore we can decompose $\rho^{n+1}=\o^{n+1}+\g^{n+1}$ unambiguously.
Note that $\F^n_\tau$ implicitly defines a Moreau-Yosida regularization $\E_\tau$ of $\E$ \cite[chapter 3]{AGS}, and more precisely $\F^n_\tau(\rho^{n+1})=\E_\tau(\rho^n)$.
As a consequence the slope estimate $|\partial \E|(\rho^{n+1})\leq \frac{\W(\rho^n,\rho^{n+1})}{\tau}$ follows from standard properties of the Moreau-Yosida regularization, see e.g. \cite[lemma 3.1.3]{AGS}.
The rest of the statement follows by Theorem~\ref{theo:slope_controls_Fisher} (the slope controls the Fisher informations and the trace condition) and Lemma~\ref{lem:props_Fisher} (properties of the Fisher information).
\end{proof}
For a momentum $m\in \M^d(\bO)$ and mass density $\rho\in \P(\bO)$ we denote the quadratic kinetic energy functional
$$
\J(m,\rho):=\int_\bO\frac{|m|^2}{\rho},
$$
which is properly defined as the usual extension of to the space of measures of the integral with density $J(M,R)=\frac{|M|^2}{R}$ (by one-homogeneity).

The following lemma gives in some sense the Euler-Lagrange optimality condition for the minimizer $\rho^{n+1}$ in each JKO step.
\begin{lem}
\label{lem:Euler-Lagrange_pi}
There exists an optimal plan $\pi^{n+1}\in \Pi(\rho^n,\rho^{n+1})$ such that the momentum $m^{n+1}\in \M^d(\bO)$ defined through
\begin{equation}
\label{eq:def_momentum_n+1}
\int _\bO \xi(y)\cdot m^{n+1}(\rd y)
\coloneqq
\iint_{\bO\times\bO}\frac{y-x}{\tau}\cdot \xi(y) \pi^{n+1}(\rd x,\rd y),
\hspace{1cm}
\forall\,\xi\in C_b(\bO)^d
\end{equation}
can be written as
\begin{equation}
\label{eq:Euler-Lagrange_mn+1}
m^{n+1}=-\nabla\o^{n+1} \L^d_\O - \nabla_\G \g^{n+1}\L^{d-1}_\G- \mathfrak n_\G^{n+1},
\end{equation}
for some orthogonal measure $\mathfrak n_\G^{n+1}\in \M^{d-1}(\G)$ supported along the normal $\nu$, i-e $\mathfrak n_\G^{n+1}=\nu(x)\cdot N^{n+1}_\G$ for some scalar measure $N^{n+1}_\G\in \M(\G)$.
Moreover we have
\begin{equation}
\label{eq:control_discrete_Io_Ig}
 \I(\o^{n+1})+\I(\g^{n+1})
\leq
\mathcal J(m^{n+1},\rho^{n+1})\leq \frac{1}{\tau^2}\W^2(\rho^n,\rho^{n+1})
\end{equation}
\end{lem}
Let us stress here that $\pi^{n+1}\in \Pi(\rho^n,\rho^{n+1})$ is a ``forward-in time'' plan.
Accordingly, the term $\frac{y-x}{\tau}$ in \eqref{eq:def_momentum_n+1} should be seen as the terminal velocity of the time-$\tau$ geodesic from $x$ to $y$, a tangent vector at $y$.
This convention for the discrete momentum is chosen so that $\partial_t\rho+\dive m=0$ in the end (after taking $\tau\to 0$), and this is consistent with the minus sign in \eqref{eq:Euler-Lagrange_mn+1} ultimately leading to forward-in time diffusion.
Let us also anticipate on the fact that the orthogonal measure $\mathfrak n_\G=\lim\limits_{\tau\to 0}\mathfrak n_\G^{n+1}$ will vanish in the end (we were not able to prove that $\mathfrak n_\G^{n+1}=0$ already at the discrete level) so the key point in \eqref{eq:Euler-Lagrange_mn+1} is to identify the interior contribution $\nabla\o$ and the tangential component $\nabla_\G \g$ of the boundary part.
\begin{proof}
 Fix $\eps>0$ and let $\G^\eps=\{x\in\O,\dist(x,\G)\leq \eps\}$ be an interior $\eps$-neighborhood of $\G=\pO$.
 We let $\O^\eps=\O\setminus\G^\eps$, which is again a smooth domain if $\eps>0$ is small enough.
 For $\eps$ small enough it is easy to choose $U^\eps:\R^+\to\R^+$ supported on $[0,\eps]$ and a smooth potential $V^\eps:\bO\to \R$ with $V^\eps(x)=U^\eps(\dist(x,\G))$  which is constant inside $\O^\eps$, only depends on the normal coordinate inside $\G^\eps$, and most importantly satisfies
 $$
 \mu^\eps \coloneqq e^{-V^\eps(x)}\L^d_\O
 \narrowcv
 \mu=\L^d_\Omega+\L^{d-1}_\G
 $$
as $\eps\to 0$ (one can simply use a standard $\eps$-mollification of the one-dimensional Dirac-delta in the normal direction).
Note that $\nabla V^\eps$ is orthogonal to the boundary in the sense that
\begin{equation}
\label{eq:nabla_V_normal}
\nabla V^\eps(x)=\nabla\left( U^\eps(\dist(x,\G)\right)=-{U^\eps}'(\dist(x,\G)\nu(P(x)),
\hspace{1.5cm}
\forall\,x\in \G^\eps
\end{equation}
where $\nu$ is the inward normal to $\G=\pO$ and $P(x)=\operatorname{Proj}_\G(x)$ is the projection on the boundary.
We consider the approximate problem
$$
\rho^\eps=\underset{\rho\in\P(\bO)}{\argmin}\left\{
 \frac{1}{2\tau}\W^2(\rho,\rho^n)+\H\left(\rho\,\vert\,\mu^\eps\right)
 \right\}.
$$
Just as in Lemma~\ref{lem:exists_minimizer_JKO} there exists a unique, absolutely continuous minimizer $\rho^\eps=\rho^\eps(x)\LO$.
By standard arguments, and because $\mu^\eps \narrowcv \mu$, the functional $\F^n_{\tau,\eps}\coloneqq \frac 1{2\tau}\W^2(\cdot,\rho^n)+\H(\cdot\,\vert\,\mu^\eps)$ Gamma-converges to the JKO functional $\F^n_\tau$ as $\eps\to 0$.
As a consequence $\rho^\eps$ converges narrowly to the unique minimizer $\rho^{n+1}$ of $\F^n_\tau$ as $\eps\to 0$.

For $\eps>0$ the minimizer $\rho^\eps$ is absolutely continuous, hence by the Brenier-McCann theorem there exists a unique (forward) optimal plan $\pi^\eps=(T^\eps,\operatorname{id})\pf \rho^{\eps}$ from $\rho^n$ to $\rho^\eps$, induced by the unique (backward) optimal map $T^\eps$ from $\rho^\eps$ to $\rho^n=T^\eps\pf \rho^\eps$.
Standard stability of optimal plans \cite[thm. 1.52]{OTAM} gives that $\pi^\eps\narrowcv\pi^{n+1}$ as $\eps\to 0$ for some optimal plan $\pi^{n+1}\in \Pi_\text{opt}(\rho^n,\rho^{n+1})$.
The momentum $m^\eps$, defined by duality as
$$
\int _\bO \xi(y)\cdot m^\eps(\rd y)
\coloneqq \iint_{\bO\times\bO}\frac{y-x}{\tau}\cdot \xi(y) \pi^\eps(\rd x,\rd y),
\hspace{1cm}
\forall\,\xi\in C_b(\bO)^d
$$
therefore converges narrowly to $m^{n+1}$ similarly defined by \eqref{eq:def_momentum_n+1}.
For $\eps>0$, by-now fairly standard arguments \cite[chapter 8]{OTAM} give the Euler-Lagrange optimality condition for the $\eps$-problem in the form
\begin{equation}
\label{eq:Euler-Lagrange_meps}
m^\eps(\rd y)=\frac{y-T^\eps(y)}{\tau}\rho^\eps(\rd y)
=-\rho^\eps\nabla\left(\log\rho^\eps+V^\eps\right)
=-\nabla\rho^\eps - \rho^\eps\nabla V^\eps
\end{equation}
with a slight abuse of notations.
We have moreover
$$
\J(m^\eps,\rho^\eps)
=\int_\bO \frac{|m^\eps|^2}{\rho^\eps}
=\int_\bO \left|\frac{T^\eps(y)-y}\tau\right|^2\rho^\eps(\rd y)
=\frac{\W^2(\rho^\eps,\rho^n)}{\tau^2}
\xrightarrow[\eps\to 0]{} \frac{\W^2(\rho^{n+1},\rho^n)}{\tau^2}.
$$
Since $(\rho^\eps,m^\eps)\narrowcv (\rho^{n+1},m^{n+1})$ this gives, by lower semi-continuity of $\mathcal J$,
$$
\J(m^{n+1},\rho^{n+1})
\leq \liminf\limits_{\eps\to 0} \J(m^{\eps},\rho^{\eps})
=\frac{\W^2(\rho^{n+1},\rho^n)}{\tau^2}
<\infty.
$$
In particular by standard properties of $\mathcal J$ we have $|m^{n+1}|\ll\rho^{n+1}=\o^{n+1}+\g^{n+1}\in L^1(\O)+L^1(\G)$, and this defines unambiguously  the interior and boundary parts
$$
m^{n+1}
=
m^{n+1}\mrest_\O +m^{n+1}\mrest_\G
\eqqcolon
m^{n+1}_\O+m^{n+1}_\G.
$$
The goal is now to identify the interior part $m^{n+1}_\Omega$ as well as the tangential part of $m^{n+1}_\G$ as claimed in \eqref{eq:Euler-Lagrange_mn+1}.
Roughly speaking, this will be achieved by testing the full momentum $m^{n+1}$ first with vector fields compactly supported in the interior, and then with tangential vector fields supported along the boundary.
One of the key points will be that our $\eps$-regularization allows to carefully maintain a certain orthogonality condition between $\nabla V^\eps$, which always points in the normal direction according to \eqref{eq:nabla_V_normal}, and a suitable extension of tangent boundary fields in a small interior neighborhood.

More precisely, recall from Lemma~\ref{lem:exists_minimizer_JKO} that $\o^{n+1},\g^{n+1}$ have $W^{1,p}$ Sobolev regularity for some $p>1$, and observe that, owing to $\mathcal J(m^{n+1},\rho^{n+1})<\infty$ and $\rho^{n+1}=\o^{n+1}+\g^{n+1}$, we have absolute continuity $m_\O^{n+1}\ll \L^d_\O$ and $m^{n+1}_{\G}\ll\L^{d-1}_\pO$.
We first claim that $m^{n+1}_\O=-\nabla \o^{n+1}\LO$.
To this end, note that $m^{n+1}$ has finite mass since
\begin{multline*}
 |m^{n+1}| =
 \left|\lim\limits_{\eps\to 0} m^\eps\right|
 \leq \liminf\limits_{\eps\to 0} |m^\eps|
 =\liminf\limits_{\eps\to 0} \left(\int_{\bO}\left|\frac{y-T^\eps(y)}{\tau}\right|\rho^\eps\right)
 \\
 \leq
 \liminf \limits_{\eps\to 0}\left(\int_{\bO}\left|\frac{T^\eps(y)-y}{\tau}\right|^2\rho^\eps\right)^{\frac 12}
 =
 \liminf\limits_{\eps\to 0} \frac{\W(\rho^\eps,\rho^n)}{\tau}
 =\frac{\W(\rho^{n+1},\rho^n)}{\tau}.
\end{multline*}
Take any vector-field $\xi\in C^1_c(\O)^d$.
Then
$$
\int_{\O}\xi \cdot m^{n+1}_\O
=
\int_{\bO} \xi \cdot m^{n+1}
=
\lim\limits_{\eps\to 0} \int_{\bO}\xi \cdot m^\eps
\overset{\eqref{eq:Euler-Lagrange_meps}}{=}
-\lim\limits_{\eps\to 0} \int_{\bO}\xi \cdot (\nabla\rho^\eps+\rho^\eps\nabla V^\eps)
=
-\lim\limits_{\eps\to 0} \int_{\bO}\xi \cdot \nabla\rho^\eps,
$$
where the last equality holds because for $\eps$ small enough $\nabla V^\eps$ is supported in $\G^\eps$ where $\xi$ vanishes (being compactly supported away from the boundary).
Integrating by parts gives
$$
\int_{\O}\xi \cdot m^{n+1}_\O
=
-\lim\limits_{\eps\to 0} \int_{\bO}\xi \cdot \nabla\rho^\eps
=\lim\limits_{\eps\to 0} \int_{\bO} \dive(\xi)\rho^\eps
=\int_{\bO} \dive(\xi)\rho^{n+1}
=\int_{\O} \dive(\xi)\ \o^{n+1},
$$
where the last equality holds again because $\xi$ is supported away from the boundary.
Since $m^{n+1}_\O$ and $\o^{n+1}$ are absolutely continuous this fully characterizes $m^{n+1}_\O=-\nabla\o^{n+1}\LO$ as claimed.

For the more delicate boundary part let us decompose into tangential and orthogonal components
$$
m^{n+1}_{\G}=\mathfrak t_\G^{n+1}+\mathfrak n_\G^{n+1}.
$$
We need to show
$$
\mathfrak t_\G^{n+1} = - \nabla_\G \gamma^{n+1}\L^{d-1}_\G.
$$
Take a tangential field $\xi\in C^1(\G)^d$.
Fix a small $\delta>0$ and take any cutoff-function $\eta:\R^+\to\R^+$ such that $\eta(0)=1$, $0\leq \eta\leq 1$, and $\eta\equiv 0$ in $[\delta,+\infty)$.
For small enough $\delta$ the normal projection on the boundary $P(x)=\operatorname{Proj}_\G(x)$ is well-defined and smooth on the tubular neighborhood $\G^\delta$.
We extend $\eta$ to $\zeta\in C^1(\bO)^d$ as $\zeta(x)\coloneqq \eta(\dist(x,\G)) \xi(P(x))$, and crucially observe that $\zeta$ remains compactly supported in $\delta$ and tangential in the sense that $\zeta(x)\cdot \nu(P(x))=0$ for all $x\in\G$.
We compute next
\begin{multline*}
 \int_\G \xi\cdot m^{n+1}_\G + \int_\O \zeta\cdot m^{n+1}_\O
 = \int_\bO \zeta\cdot m^{n+1}
 =\lim\limits_{\eps\to 0} \int_\bO \zeta\cdot m^{\eps}
 =
 -\lim\limits_{\eps\to 0} \int_\bO \zeta\cdot (\nabla\rho^\eps+\rho^\eps \nabla V^\eps).
\end{multline*}
Now, for fixed $\delta$ and small enough $\eps<\delta$, we have owing to \eqref{eq:nabla_V_normal} that $\zeta\cdot \nabla V^\eps=0$ inside $\G^\eps$ (where $\nabla V^\eps$ is normal and $\zeta$ is tangential due to $\eps<\delta$), while $\nabla V^\eps$ vanishes identically outside of $\G^\eps $.
Hence $\zeta\cdot \nabla V^\eps\equiv 0$ in $\bO$, and since on the boundary $\zeta|_\pO=\xi$ is tangential we can integrate by parts the term $\int_\O \zeta\cdot\nabla\rho^\eps$ for any fixed $\eps>0$ to get
\begin{multline*}
 \int_\G \xi\cdot m^{n+1}_\G + \int_\O \zeta\cdot m^{n+1}_\O
 =-\lim\limits_{\eps\to 0} \int_\bO \zeta\cdot \nabla\rho^\eps
 =\lim\limits_{\eps\to 0} \int_\bO \dive(\zeta)\rho^\eps
 \\
 =\int_\bO \dive(\zeta)\rho^{n+1}
 = \int_\G \dive(\zeta)\gamma^{n+1} + \int_\O \dive(\zeta)\o^{n+1}.
\end{multline*}
Because we just proved that $m_\O^{n+1}=\nabla\o^{n+1}$, and since $\zeta|_\pO\cdot \nu=\xi\cdot \nu=0$ on the boundary, we can cancel out $\int_\O \zeta\cdot m^{n+1}_\O =-\int_\O \zeta\cdot \nabla \o^{n+1}= \int_\O \dive(\zeta)\,\o^{n+1}$ in the previous equality and we are left with
$$
\int_\G \xi\cdot m^{n+1}_\G
 = \int_\G \dive(\zeta)\gamma^{n+1}.
$$
Since the extension $\zeta$ is actually tangential in a neighborhood of $\pO$ the full divergence $\dive(\zeta)|_{\pO}$ matches the intrinsic tangential divergence $\dive_\G (\xi)$, hence for any tangential field $\xi\in C^1(\G)$
$$
\int_\G \xi\cdot \mathfrak t_\G^{n+1}
=
\int_\G \xi\cdot m^{n+1}_\G
 = \int_\G \dive_\G(\xi)\ \g^{n+1}.
$$
Because $\G$ has no boundary this fully characterizes $\mathfrak t_\G^{n+1}=-\nabla_\G\g^{n+1}$.

Finally, since we known now that $-m^{n+1}=\nabla\o^{n+1} \L^d_\O + \nabla_\G \g^{n+1}\L^{d-1}_\G+ \mathfrak n_\G^{n+1}$
we get for free that
\begin{multline*}
 \I(\o^{n+1})+\I(\g^{n+1})
=
\int_\O \frac{|\nabla\o^{n+1}|^2}{\o^{n+1}} + \int_\G \frac{|\nabla_\G\g^{n+1}|^2}{\g^{n+1}}
\\
\leq
\int_\O \frac{|\nabla\o^{n+1}|^2}{\o^{n+1}} + \int_\G \frac{|\mathfrak n_\G^{n+1}|^2}{\g^{n+1}} + \int_\G \frac{|\nabla_\G\g^{n+1}|^2}{\g^{n+1}}
\\
=
\int_{\bO} \frac{|m^{n+1}|^2}{\rho^{n+1}}
=\J(m^{n+1},\rho^{n+1})
\end{multline*}
and the proof is complete.
\end{proof}

\begin{lem}
\label{lem:discrete max}
 Given any $\rho^n\in \P(\bO)$ the unique minimizer $\rho^{n+1}$ defined in Lemma~\ref{lem:exists_minimizer_JKO} satisfies:
 \begin{itemize}
 \item if $\rho^n\geq \underline{c}\mu$ for some positive constant $\underline{c}$, then  $\rho^{n+1}\geq \underline{c}\mu$;
 \item if $\rho^n\leq \overline{c}\mu$ for some positive constant $\overline{c}$, then  $\rho^{n+1}\leq \overline{c}\mu$.
 \end{itemize}
\end{lem}

\begin{proof}
We consider an approximation similar to that in Lemma \ref{lem:Euler-Lagrange_pi}.
Yet, not only we want to regularize $\mu$ into $\mu^\varepsilon$, but we also need to modify $\rho^n$ into a new measure $\rho^{n,\varepsilon}$ so that we preserve the inequality $\rho^{n,\varepsilon} \geq \underline{c}\mu^\varepsilon$ (resp. $\rho^{n,\varepsilon} \leq \overline{c}\mu^\varepsilon$).
To do so, we can simply apply the heat semigroup in the domain $\Omega$ to both $\mu$ and $\rho^n$.
Note that this construction was not used in Lemma \ref{lem:Euler-Lagrange_pi} because we needed a certain structure in the approximation, and in particular we wanted to guarantee $\mu^\varepsilon = e^{-V_\varepsilon}$ for a potential $V_\varepsilon$ which is function of the distance to $\partial\Omega$, constant far enough from $\partial\Omega$, which does not happen for the heat flow.
On the other hand, we did not need in Lemma \ref{lem:Euler-Lagrange_pi} to modify $\rho^n$.

We then observe that we can apply \cite[Lemma 2.4]{IacPatSan} which concerns one step of the JKO scheme for a functional of the form $\rho\mapsto \int F(\frac{\rd\rho}{\rd\mu})\rd\mu$.
We apply this to the regularized JKO functional where $\mu$ is replaced by $\mu^\varepsilon$ because \cite{IacPatSan} did not explicitly mention the case where $\mu$ is singular.
We obtain the inequalities $\rho^\varepsilon\geq \underline{c}\mu^\varepsilon$ (resp., $\rho^\varepsilon\leq \overline{c}\mu^\varepsilon$) and we can take the limit $\varepsilon\to 0$ (the $\eps$-JKO functional is easily checked to $\Gamma$-converge, hence the minimizers converge as well).
\end{proof}

In order to retrieve later on a weak solution in the limit $\tau\to 0$ we define three very classical time interpolants, based on the discrete solutions $(\rho^n)_{n\geq 0}$.
Writing for simplicity $t^n=n\tau$, the piecewise-constant interpolation $\orho^\tau:[0,\infty)\to \P(\bO)$ is
\begin{equation}
 \orho^\tau_t\coloneqq \rho^{n+1},
 \hspace{1cm}\mbox{for }t\in (t^n,t^{n+1}].
\end{equation}
The piecewise geodesic interpolation $\rho^\tau:[0,\infty)\to \P(\bO)$ is defined as
\begin{equation}
\label{eq:def_rho_tau_geodesic}
 \rho^\tau_t\coloneqq \left(\frac{t^{n+1}-t}{\tau}x+\frac{t-t^n}{\tau}y\right)\pf\pi^{n+1},
 \hspace{1cm}\mbox{for }t\in [t^n,t^{n+1}],
\end{equation}
where $\pi^{n+1}\in \Pi(\rho^{n},\rho^{n+1})$ is the optimal plan from Lemma~\ref{lem:Euler-Lagrange_pi}.
Finally, we also define \emph{De~Giorgi's variational interpolant} $\trho^\tau:[0,\infty)\to \P(\bO)$ as
\begin{equation}
\label{eq:def_DeGiorgi_interpolant}
\trho^\tau_t
\coloneqq
\underset{\rho\in\P(\bO)}{\argmin}\left\{
\frac{1}{2r}\W^2(\rho,\rho^n) + \E(\rho)
\right\}
\hspace{1cm}\mbox{for }t=t^n+r\in (t^n,t^{n+1}],\quad r\in (0,\tau].
\end{equation}
This is well-defined:
Arguing exactly as in Lemma~\ref{lem:exists_minimizer_JKO} there is a unique minimizer for all $r>0$.
Note that for $t=t^{n+1}$ we have $r=\tau$, in which case $\trho^\tau(t^{n+1})=\rho^{n+1}$ is nothing but the next JKO step.
Standard properties of the Moreau-Yosida regularization \cite[lemma 3.1.2]{AGS} guarantee that $\trho^\tau_t\to \rho^n$ as $t\searrow t^n\Leftrightarrow r\searrow 0$ in \eqref{eq:def_DeGiorgi_interpolant}, thus $\trho^\tau_t$ indeed interpolates continuously between successive steps $\rho^n,\rho^{n+1}$ of the JKO minimizing movement.
Note that by definition all three interpolants also match at discrete times, i-e $\rho^\tau(t^n)=\orho^\tau(t^n)=\trho^\tau(t^n)=\rho^n$.
Moreover, since $\rho^\tau$ is piecewise geodesic we have that the metric speed
$$
|\dot\rho^\tau_t|^2=cst=\frac{\W^2(\rho^n,\rho^{n+1})}{\tau^2}
\hspace{1cm}\mbox{for }t\in [t^n,t^{n+1}].
$$
Our last a-priori estimate is completely standard for minimizing movements in metric spaces and is just a discrete Energy Dissipation Inequality.
Note that this involves both the geodesic interpolation $\rho^\tau$ and De Giorgi's interpolant $\trho^\tau$.
\begin{lem}
\label{lem:discrete_JKO_estimates_T}
For any $\tau>0$ and $t^N=N\tau$, $N\in\N$, there holds
\begin{equation}
\label{eq:discrete_EDI}
 \E\left(\rho^\tau(t^N)\right) + \int_0^{t^N}\left(\frac 12\left|\dot \rho^\tau_t\right|^2 + \frac 12\left|\partial\E\right|^2(\trho^\tau_t)\right)\rd t
 \leq \E\left(\rho_0\right).
\end{equation}
In particular since $\E\geq 0$ one has
\begin{equation}
\label{eq:uniform_speed_slope_estimate}
\sup\limits_{\tau\to 0}\int_0^{\infty}\left(\frac 12\left|\dot \rho^\tau_t\right|^2 + \frac 12\left|\partial\E\right|^2(\trho^\tau_t)\right)\rd t \leq \E(\rho_0)<+\infty.
\end{equation}
\end{lem}
\begin{proof}
For fixed $n\in \N$, standard properties of the Moreau-Yosida approximation give the one-step dissipation estimate
$$
\H(\rho^{n+1})+ \frac {\W^2(\rho^n,\rho^{n+1})}{2\tau} + \frac 12\int_{t^n}^{t^{n+1}}\left|\partial\E\right|^2(\trho^\tau_t)\rd t
\leq \E(\rho^n),
$$
see e.g. \cite[lemma 7.5 and eq. (7.16)]{peletier2022jump} or \cite[theorem 3.1.4 with lemma 3.1.3]{AGS}.
Recall that by definition $\rho^\tau$ is the geodesic interpolation and therefore has constant speed $|\dot\rho^\tau_t|^2=cst =\frac{\W^2(\rho^n,\rho^{n+1})}{\tau^2}$.
Summing the previous inequality from $n=0$ to $n=N-1$, we get a telescopic sum for the $\E\left(\rho^{n+1}\right)-\E\left(\rho^n\right)$ term with by definition $\rho^N=\rho^\tau(t^N)$ and $\rho^0=\rho_0$, and \eqref{eq:discrete_EDI} follows.
\end{proof}


\section{Convergence to a dissipative weak solution}
\label{sec:convergence}
In this section we fully exploit the discrete estimates from Section~\ref{sec:discrete_estimates} in order to prove our main result.
For the ease of exposition we split the proof into separate statements.
 Our first result is a rather soft and generic convergence statement
\begin{prop}
\label{prop:abstract_convergence_curve}
There exists a discrete sequence $\tau\to 0$ (not relabeled for simplicity) and a continuous curve $\rho:[0,\infty)\to\P(\bO)$ with $\rho(0)=\rho_0$ such that
\begin{enumerate}
 \item
the interpolants $\rho^\tau,\trho^\tau,\orho^\tau$ converge locally uniformly towards the same limit $\rho$
 $$
 \sup\limits_{t\in[0,T]}\Big[\W(\rho_t,\orho_t^\tau)+ \W(\rho_t,\rho_t^\tau)+\W(\rho_t,\trho_t^\tau)\Big]\xrightarrow[\tau\to 0]{}0
 $$
 in any finite time interval $[0,T]$.
 \item
 we have the Energy Dissipation Inequality
 \begin{equation}
 \label{eq:EDI}
 \E(\rho_T)+\int_0^T\left(\frac 12\left|\dot \rho_t\right|^2 + \frac 12\D(\rho_t)\right)\rd t \leq \E(\rho_0),
 \hspace{1cm}\forall\,T>0.
 \end{equation}
\end{enumerate}
\end{prop}
\begin{proof}
First, owing to the Prokhorov theorem, $\P(\bO)$ is relatively compact for the narrow convergence.
Since in bounded domains the Wasserstein distance metrizes the narrow convergence \cite[thm 6.9]{villani_BIG} we have that $\{\rho^\tau_t\}_{\tau>0}$ lies in a fixed $\W$-relatively compact set for any fixed $t\geq 0$.
By \eqref{eq:uniform_speed_slope_estimate} we also have $\sup\limits_{\tau\to 0}\int_0^T |\dot\rho^\tau_t|^2\rd t<+\infty$ for any fixed $T>0$.
In particular $\W(\rho^\tau_t,\rho^\tau_s)\leq C_T|t-s|^\frac12$ in any finite interval.
By the Ascoli-Arzel\'a theorem there is a discrete sequence such that $\rho^\tau$ converges $\W$-uniformly to a limit curve $\rho$.
By diagonal extraction if needed we can assume that the limit does not depend on the finite horizon $T>0$, hence $\rho^\tau\to\rho$ locally uniformly.

Fix now any $t\geq 0$, and let $n=\lfloor t/\tau\rfloor$ such that $t=t^n+r\in [t^n,t^{n+1})$ for some $r\in [0,\tau)$.
Testing $\rho=\rho^n$ as a competitor in \eqref{eq:def_DeGiorgi_interpolant} and recalling that $\E(\tilde\rho^\tau_t)\geq 0$, we get the rough bound
$$
\frac{1}{2r}\W^2(\trho^\tau_t,\rho^n)
\leq
\E(\trho^\tau_t)+\frac{1}{2r}\W^2(\trho^\tau_t,\rho^n) \leq \E(\rho^n),
$$
so that
$$
\W^2(\trho^\tau_t,\rho^n) \leq 2r \E(\rho^n)\leq 2\tau \E(\rho_0).
$$
Similarly testing $\rho=\rho^n$ in \eqref{eq:JKO}, and because $\rho^\tau$ interpolates geodesically between $\rho^n,\rho^{n+1}$, one gets
$$
\W^2(\rho^\tau_t,\rho^n)
\leq
\W^2(\rho^{n+1},\rho^n)
\leq
2\tau \E(\rho^n)
\leq 2\tau \E(\rho_0).
$$
By triangular inequality we get
$$
\W(\trho^\tau_t,\rho^\tau_t)
\leq
\W(\trho^\tau_t,\rho^n)+\W(\rho^n,\rho^\tau_t)\leq 2\sqrt{2\tau\E(\rho_0)}\to 0
$$
and therefore $\trho^\tau$ also converges uniformly to the same limit $\rho$ as $\rho^\tau$.
For the same reason
$$
\W(\orho^\tau_t,\rho^\tau_t)
=
\W(\rho^{n+1},\rho^\tau_t)
\leq
\W(\rho^{n+1},\rho^n)
\leq
\sqrt{2\tau\E(\rho_0)}\to 0
$$
and thus $\orho^\tau$ also converges uniformly to $\rho$.
This common limit trivially satisfies the initial condition in the sense that $\rho(0)=\lim \rho^\tau(0)=\lim \rho_0=\rho_0$.

Finally let us turn into the dissipation inequality \eqref{eq:EDI}.
For fixed $T>0$ let $N=\lfloor T/\tau\rfloor$ and $t^N=N\tau$, so that $T\in [t^N,t^{N+1})$.
By continuity of $t\mapsto \rho_t$ and arguing exactly as before we have $\W(\rho_T,\rho^\tau_{t^{N+1}})\leq \W(\rho_T,\rho^\tau_T)+\W(\rho^\tau_T,\rho^\tau_{t^{N+1}})\to  0$ and therefore by lower semi-continuity
\begin{equation}
\label{eq:H_leq_liminf_HN}
\E(\rho_T)\leq \liminf\limits_{\tau\to 0} \E(\rho^\tau_{t^{N+1}}).
\end{equation}
For the dissipation terms, we note first from $T\leq t^{N+1}$ and $\D\leq |\partial\E|^2$ in Theorem~\ref{theo:slope_controls_Fisher} that
$$
\int_0^{T}\left(\frac 12\left|\dot \rho^\tau_t\right|^2 + \frac 12\D(\trho^\tau_t)\right)\rd t
\leq
\int_0^{t^{N+1}}\left(\frac 12\left|\dot \rho^\tau_t\right|^2 + \frac 12\left|\partial\E\right|^2(\trho^\tau_t)\right)\rd t
$$
simply because $T\leq t^{N+1}$.
By lower semicontinuity of the $H^1$ action $\rho\mapsto\int |\dot\rho_t|^2\rd t$ w.r.t. uniform convergence (as a supremum of continuous functions) we have that
\begin{equation}
\label{eq:speed_leq_liminf_speed}
\int_0^T|\dot\rho_t|^2\rd t
\leq
\liminf \limits_{\tau\to 0}
\int_0^T|\dot\rho^\tau_t|^2\rd t
\leq
\liminf \limits_{\tau\to 0}
\int_0^{t^{N+1}}|\dot\rho^\tau_t|^2\rd t.
\end{equation}
By Fatou's lemma and the lower semicontinuity of $\D$ (Theorem~\ref{theo:slope_controls_Fisher}) we get
\begin{equation}
 \label{eq:slope_leq_liminf_slope}
\int _0^T \D(\rho_t)\rd t
\leq
\int _0^T \liminf\limits_{\tau\to 0} \D(\rho^\tau_t)\rd t
\leq
\liminf \limits_{\tau\to 0}\int _0^T \D(\rho^\tau_t)\rd t
\leq
\liminf \limits_{\tau\to 0}\int _0^{t^{N+1}} |\partial\E|^2(\rho^\tau_t)\rd t
\end{equation}
Gathering \eqref{eq:H_leq_liminf_HN}\eqref{eq:speed_leq_liminf_speed}\eqref{eq:slope_leq_liminf_slope}, we finally retrieve
\begin{multline*}
\E(\rho_T)+\int_0^T\left(\frac 12\left|\dot \rho_t\right|^2 + \frac 12\D(\rho_t)\right)\rd t
\\
\leq
\liminf\limits_{\tau\to 0} \E(\rho^\tau_{t^{N+1}})
+
\frac 12\liminf \limits_{\tau\to 0}
\int_0^{t^{N+1}}|\dot\rho^\tau_t|^2\rd t
+
\frac 12 \liminf\limits_{\tau\to 0} \int _0^{t^{N+1}} |\partial\E|^2(\rho^\tau_t)\rd t
\\
\leq
\liminf\limits_{\tau\to 0} \left\{\E(\rho^\tau_{t^{N+1}})
+
\int_0^{t^{N+1}}\left(\frac 12|\dot\rho^\tau_t|^2\rd t
+
\frac 12 |\partial\E|^2(\rho^\tau_t)\right)\rd t
\right\}
\overset{\eqref{eq:discrete_EDI}}{\leq }
\E(\rho_0)
\end{multline*}
and the proof is complete.
\end{proof}
So far this proof of convergence relied solely on soft and fairly general metric considerations, with the exception that $\D\geq |\partial\E|^2$ from Theorem~\ref{theo:slope_controls_Fisher}.
The dissipation estimate \eqref{eq:EDI}, however, was retrieved in terms of the explicit functional $\D(\rho)=\I(\o)+\I(\g)$ rather than the less tractable slope $|\partial\E|(\rho)$.
Since we could not prove that $\D=|\partial\E|^2$ in Theorem~\ref{theo:slope_controls_Fisher}, and because we could not establish a general upper chain rule, this is where purely metric arguments become limited.
In order to retrieve the PDE we thus have to resort to more direct arguments, mainly based on the Euler-Lagrange optimality condition from Lemma~\ref{lem:Euler-Lagrange_pi}.

For technical reasons we first need to upgrade the narrow convergence $\orho^\tau_t\narrowcv \rho_t$ to separate convergence of the interior and boundary parts.
\begin{lem}
\label{lem:separate_convergence_o_g}
For the same discrete sequence $\tau\to 0$ as in Proposition~\ref{prop:abstract_convergence_curve} there holds
\begin{equation*}
\oo^\tau_t\rightharpoonup\o_t\mbox{ in }L^1(\O)
\qtext{and}
\og^\tau_t\rightharpoonup\g_t\mbox{ in }L^1(\G)
\hspace{1cm}\mbox{for all }t\geq 0,
\end{equation*}
\end{lem}
\begin{proof}
Fix any $t\geq 0$.
From the entropy bounds $\H(\oo^\tau_t)+\H(\og^\tau_t)=\E(\bar\rho^\tau_t)\leq \E(\rho_0)$ in \eqref{eq:discrete_EDI}, for any subsequence $\tau_k$ there is a sub-subsequence $\tau_{k_l}$ such that $\oo^{\tau_{k_l}}_t\rightharpoonup u$ and $\og^{\tau_{k_l}}_t\rightharpoonup v$ in $L^1(\O),L^1(\G)$  as $l\to\infty$, for some limits $u,v$.
But since we already know that $\orho^\tau_t\narrowcv \rho_t$ this shows that any such limits are necessarily $u=\rho_t\mrest\O=\o_t$ and $v=\rho_t\mrest\pO=\g_t$.
This classically implies that the whole sequence converges and concludes the proof.
\end{proof}

\begin{prop}
\label{prop:rho_weak_sol}
The curve $\rho=\o+\g$ from Proposition~\ref{prop:abstract_convergence_curve} solves
$$
\begin{cases}
 \partial_t\omega=\Delta\omega & \mbox{in }\Omega\\
 \omega=\gamma & \mbox{on }\partial\Omega\\
 \partial_t\gamma=\Delta_\G\gamma-\partial_\nu\omega  & \mbox{in }\pO
\end{cases}
$$
in the sense that $\o_t|_{\pO}=\g_t$ for a.e. $t\geq 0$, $\o\in L^1_{loc}([0,\infty);W^{1,p}(\O))$ and $\g\in L^1_{loc}([0,\infty);W^{1,q}(\G))$ with $p,q>1$ given by \eqref{eq:Fisher_mass_control_W1p} (with respectively $d\geq 2$ and $d-1\geq 1$), and
$$
\int_{\bO}\varphi\rho_{t_1} -\int_{\bO}\varphi\rho_{t_0}
+ \int_{t_0}^{t_1}\int_{\O}\nabla\varphi\cdot \nabla\o
+ \int_{t_0}^{t_1}\int_{\G}\nabla_\G\varphi\cdot \nabla_\G\g
=0
\hspace{1cm}\forall\,\varphi\in C^1(\bO),\quad\forall\,0\leq t_0<t_1<\infty.
$$
\end{prop}
\begin{proof}
We first observe from the energy dissipation \eqref{eq:EDI} that $\int_0^T\D(\rho_t)dt<\infty$.
By Theorem~\ref{theo:slope_controls_Fisher} this immediately implies $\o_t|_{\pO}=\g_t$ for a.e. $t\geq 0$ and, combined with \eqref{eq:Fisher_mass_control_W1p}, also entails the claimed Sobolev regularity for $\o,\g$.

Let us now focus on the weak formulation of the PDE.
Recalling that the piecewise geodesic interpolation $\rho^\tau_t$ is defined by \eqref{eq:def_rho_tau_geodesic}, we naturally define the corresponding momentum $ m^\tau_t\in \M(\bO)^d$ as follows.
Let $\pi^{n+1}\in \Pi(\rho^n,\rho^{n+1})$ be the optimal plan from Lemma~\ref{lem:Euler-Lagrange_pi}.
For $x,y\in \bO$ we denote by
$$
z_t^{xy}\coloneqq \frac{t^{n+1}-t}{\tau} x + \frac{t-t^n}{\tau} y,
\hspace{1cm}t\in [t^n,t^{n+1}]
$$
the time-$\tau$ geodesic from $x$ to $y$.
Setting
$$
\label{eq:def_momentum_mtau}
\int _\bO \xi(z)\cdot m^{\tau}_t(\rd z)
\\ \coloneqq \iint_{\bO\times\bO}\frac{y-x}{\tau}\cdot \xi\left(z_t^{xy}\right) \pi^n(\rd x,\rd y),
\hspace{1cm}
\forall\,\xi\in C_b(\bO)^d
$$
and $m^\tau\coloneqq m^\tau_t\rd t\in \M^d\left([0,\infty)\times \bO\right)$, a classical computation shows that $(\rho^\tau,m^\tau)$ solve the continuity equation $\partial_t \rho^\tau +\dive m^\tau=0$ on $(0,\infty)\times \bO$ with no flux boundary conditions, i.e.
\begin{equation}
\label{eq:continuity_equation_rhotau_mtau}
\int_{\bO}\varphi\,\rho^\tau_{t_1} -\int_{\bO}\varphi\,\rho^\tau_{t_0}
+ \int_{t_0}^{t_1}\int_{\bO}\nabla\varphi\cdot m^\tau
=0
\hspace{1cm}\forall\,\varphi\in C^1(\bO),\quad\forall\ 0\leq t_0<t_1.
\end{equation}
Moreover for any fixed $t\in [t^n,t^{n+1}]$ we have
\begin{multline*}
 \left|\int _\bO \xi(z)\cdot m^{\tau}_t(\rd z)\right|
\leq
\left(\iint_{\bO\times\bO}\left|\frac{y-x}{\tau}\right| \pi^n(\rd x,\rd y)\right)\|\xi\|_{L^\infty(\bO)}
\\
\leq
\left(\iint_{\bO\times\bO}
\frac{|y-x|^2}{\tau^2}
\pi^n(\rd x,\rd y)\right)^{\frac 12}\|\xi\|_{L^\infty(\bO)}
\\
=
\frac{\W(\rho^n,\rho^{n+1})}{\tau}\|\xi\|_{L^\infty(\bO)}
=
\left|\dot\rho^\tau_t\right|\  \|\xi\|_{L^\infty(\bO)}.
\end{multline*}
As a consequence in any finite time interval $Q_T=[0,T]\times\bO$ we have
$$
\left |m^\tau\right|(Q_T)
=
\int_0^T\left|m^\tau_t\right|(\bO)\, \rd t
\leq
\int_0^T \left|\dot\rho^\tau_t\right|\, \rd t
\leq \sqrt T \left(\int_0^T \left|\dot\rho^\tau_t\right|^2\rd t\right)^{\frac 12}
\leq C_T
$$
uniformly in $\tau>0$, where the last inequality follows from \eqref{eq:uniform_speed_slope_estimate}.
By standard compactness and diagonal extraction this gives convergence $m^\tau\narrowcv m$ narrowly as $\tau \to 0$ in any finite time interval for some limit $m$ (and up to subsequences if needed).
Moreover the weak formulation \eqref{eq:continuity_equation_rhotau_mtau} for $\partial_t \rho^\tau+\dive m^\tau=0$ immediately passes to the limit and $(\rho,m)$ therefore also solve the same continuity equation.

Similarly, if $m^{n+1}$ is the discrete momentum from Lemma~\ref{lem:Euler-Lagrange_pi}, it is easy to see that the piecewise constant momentum $\overline m^\tau=\overline m^\tau_t\rd t$ with
$$
\overline m^\tau_t \coloneqq m^{n+1}
\hspace{1cm}\mbox{for }t\in (t^n,t^{n+1}]
$$
satisfies $|\overline m^\tau|(Q_T)\leq C_T$ and therefore also converges to some $\overline m$.
Exactly as in \cite[lemma 8.9]{OTAM} we claim that $\overline m=m$.
Indeed, fix any Lipschitz function $\xi:\bO\to \R^d$ and $t\in [0,T]$.
For $t\in (t^n,t^{n+1}]$ the time-$\tau$ geodesic $z^{xy}_t$ between $x,y$ satisfies of course $|z_t^{xy}-y|\leq |x-y|$, hence by definition of $m$ and $\overline m$
\begin{multline*}
\left|\int_\bO \xi\cdot\left[m^\tau_t-\overline m^\tau_t\right]\right|
=
\left|\iint_{\bO\times\bO} \frac{y-x}{\tau}\cdot [\xi(z^{xy}_t)-\xi(y)]\pi^{n+1}(\rd x,\rd y)\right|
\\
\leq \operatorname{Lip}(\xi)\iint_{\bO\times\bO} \frac{|y-x|^2}{\tau}\pi^{n+1}(\rd x,\rd y)
\\
=
\operatorname{Lip}(\xi) \frac{\W^2(\rho^n,\rho^{n+1})}{\tau}
=\tau\operatorname{Lip}(\xi) |\dot\rho^\tau_t|^2.
\end{multline*}
As a consequence for any Lipschitz function $\zeta:[0,T]\times \bO\to \R^d$ we have
$$
\left|\int_0^T\int_\bO \zeta\cdot\left[m^\tau-\overline m^\tau\right]\right|
\leq \tau\operatorname{Lip}(\zeta)\int_0^T |\dot\rho^\tau_t|^2\,\rd t\leq C\tau\to 0,
$$
and since $m^\tau\narrowcv m,\overline m^\tau\narrowcv \overline m$ the claim follows.

Finally, let us show that $ m=-\nabla\o\cdot\LO- \nabla_\G\g\cdot\LG$.
For expediency we denote $m_\O=m\mrest [0,\infty)\times \O$ and $m_\G=m\mrest [0,\infty)\times \O$.
We first deal with the interior part.
By Lemma~\ref{lem:Euler-Lagrange_pi} we have, for fixed $0\leq t_0<t_1$ and any $\xi\in C^1_c(\O)$ compactly supported away from the boundary,
\begin{multline*}
\int_{t_0}^{t_1}\int_\O \xi\cdot m_\O
=
\int_{t_0}^{t_1}\int_\bO \xi\cdot m
=
\lim\limits_{\tau\to 0}\int_{t_0}^{t_1}\int_\bO \xi\cdot\overline m^\tau
\\
=
\lim\limits_{\tau\to 0}\int_{t_0}^{t_1}\int_\O \xi\cdot\overline m^\tau_\O
\overset{\eqref{eq:Euler-Lagrange_mn+1}}{=}
-\lim\limits_{\tau\to 0}\int_{t_0}^{t_1}\int_\O \xi\cdot\nabla \oo^\tau
\\
=
\lim\limits_{\tau\to 0}\int_{t_0}^{t_1}\int_\O \dive(\xi)\oo^\tau
=
\int_{t_0}^{t_1}\int_\O \dive(\xi)\o,
\end{multline*}
where the last equality is a straightforward application of Lebesgue's dominated convergence theorem together with the pointwise weak $L^1$ convergence $\oo^\tau_t\rightharpoonup \o_t$ for all $t\geq 0$ from Lemma~\ref{lem:separate_convergence_o_g} (and the bound $|\int_\O\dive(\xi)\oo^\tau_t|\leq \|\dive \xi\|_{\infty}$ uniformly in $t$).
Since $t_0<t_1$ and $\xi\in C^1_c(\O)$ were arbitrary this identifies the interior part $m_\O=-\nabla\o$ at least in the distributional sense.

In order to identify now the boundary contribution $m_\G=-\nabla_\G\g$, observe first that $\partial_t\rho +\dive \overline m=\partial_t\rho +\dive m=0$ with no-flux boundary condition.
One should thus expect $m_\G$ to be tangential, otherwise mass would either exit or enter the domain through $\G$.
However the no-flux condition is enforced only in a weak form, similar to \eqref{eq:continuity_equation_rhotau_mtau} with $\rho,m$ in place of $\rho^\tau,m^\tau$.
The argument therefore needs some care and for simplicity we defer it to Proposition~\ref{prop:m_boundary_tangential}.
Thus taking for granted that $m_\G$ is tangential, we only need to show equality $m_\G=-\nabla_\G\g$ in duality with tangential vector-fields.
To this end take any such tangential $\xi\in C^1(\G)$, and extend it to $\zeta\in C^1(\bO)$.
We have now
\begin{multline*}
\int_{t_0}^{t_1}\int_\G\xi\cdot m_\G
=
\int_{t_0}^{t_1}\int_\G\zeta|_{\G}\cdot m_\G
=
\int_{t_0}^{t_1}\int_{\bO}\zeta\cdot m -  \int_{t_0}^{t_1}\int_{\O}\zeta\cdot m_\O
\\
=
\lim\limits_{\tau\to 0} \int_{t_0}^{t_1}\int_{\bO}\zeta\cdot \overline m^\tau
-
\int_{t_0}^{t_1}\int_{\O}\zeta\cdot \nabla \o
=
\lim\limits_{\tau\to 0} \int_{t_0}^{t_1}\int_{\bO}\zeta\cdot \overline m^\tau
+
\int_{t_0}^{t_1}\int_{\O}\dive(\zeta) \o
\\
=
\lim\limits_{\tau\to 0}\left\{ \int_{t_0}^{t_1}\int_{\bO}\zeta\cdot \overline m^\tau
+
\int_{t_0}^{t_1}\int_{\O}\dive(\zeta) \oo^\tau\right\}
\\
\overset{\eqref{eq:Euler-Lagrange_mn+1}}{=}
\lim\limits_{\tau\to 0}\left\{
-\int_{t_0}^{t_1}\int_{\bO}\zeta\cdot \left[\nabla\oo^\tau \LO+\nabla_\G\og^\tau \LG +\overline{\mathfrak n}_\G^\tau\right]
+
\int_{t_0}^{t_1}\int_{\O}\dive(\zeta) \oo^\tau
\right\}
\\
=
-\lim\limits_{\tau\to 0}\int_{t_0}^{t_1}\int_{\G}\xi\cdot \nabla_\G\og^\tau
=
\lim\limits_{\tau\to 0}\int_{t_0}^{t_1}\int_{\G}\dive_\G(\xi)\og^\tau
=
\int_{t_0}^{t_1}\int_{\G}\dive_\G(\xi)\g.
\end{multline*}
The last equality follows again from the pointwise $L^1$ convergence $\og^\tau_t\rightharpoonup \g_t$ in $L^1(\G)$ for all $t\geq 0$ from Lemma~\ref{lem:separate_convergence_o_g} and Lebesgue's dominated convergence.
Here we also heavily exploited that $\zeta|_\G=\xi$ is tangential on the boundary in order to suitably integrate by parts in both directions, as well as to disregard the normal part $\zeta\cdot \overline{\mathfrak n}_\G^\tau=0$ of $\overline m^\tau$ along the boundary.
This identifies $m_\G=-\nabla_\G\g$ in the distributional sense and the proof is complete.
\end{proof}

In order to fully establish our main result we are still missing the long-time convergence \eqref{eq:long_time} and the propagation of initial pointwise bounds.
Observe first that, since $\bar\mu=\frac{1}{\mu(\bO)}\mu$ is a simple renormalization to unit mass, our driving  entropy is just a vertical shift $\E(\rho)=\H(\rho\,\vert\,\mu)=\H(\rho\,\vert\,\bar\mu)+C_\mu$ and the Energy dissipation \eqref{eq:EDI} holds with $\H(\rho_T\,\vert\,\bar\mu)$ and $\H(\rho_0\,\vert\,\bar\mu)$ in place of $\E(\rho_T),\E(\rho_0)$.
Moreover, Proposition~\ref{prop:rho_weak_sol} implies in particular that $\rho$ solves the continuity equation $\partial_t\rho+\dive m=0$ with driving momentum $m_t=\nabla \o_t\cdot \LO + \nabla_\G\g_t\cdot \LG$.
By standard characterization of $\W_2$-absolutely continuous curves \cite{AGS} we see that $\frac 12|\dot\rho_t|^2\leq \J(m_t,\rho_t)=\frac 12\int_{\bO}\frac{|m_t|^2}{\rho_t}=\frac 12 \int_\O\frac{|\nabla\o_t|^2}{\o_t}+\frac 12 \int_\G\frac{|\nabla_\G\g_t|^2}{\g_t}=\frac 12 \D(\rho_t)$ for a.e. $t$.
As a consequence we have
$$
\H(\rho_T\,\vert\,\bar\mu)+\int_0^T \D(\rho_t)\rd t\leq \H(\rho_0\,\vert\,\bar\mu),
\qquad \forall \,T\geq 0.
$$
We exploit next a \emph{boundary trace Logarithmic-Sobolev inequality} from \cite[section 4]{bormann2023functional} , which can be phrased in our precise framework as
$$
\H(\rho\,\vert\,\bar\mu)\leq \frac 1{2\lambda} \D(\rho),
\hspace{1cm}\forall\,\rho\in\P(\bO)
$$
for some $\lambda\equiv\lambda(\O)>0$.
This shows that
$$
\H(\rho_T\,\vert\,\bar\mu)+2\lambda \int_0^T \H(\rho_t\,\vert\,\bar\mu)\rd t\leq \H(\rho_0\,\vert\,\bar\mu),
\qquad \forall \,T\geq 0
$$
and immediately entails the first exponential decay in \eqref{eq:long_time} by Gr\"onwall's lemma.
The decay in total variation is as usual a consequence of the general Csisz\'ar-Kullback-Pinsker inequality $|a-b|_{TV}\leq \sqrt{\frac{1}{2}\H(a\,\vert\,b)}$ for arbitrary probability distributions $a,b$.
Finally, the propagation of pointwise bounds $\rho_0\leq \overline c\mu$ (resp. $\rho_0\geq \underline c\mu$) is an easy consequence of Lemma \ref{lem:discrete max}, iterated along the JKO scheme.

\begin{appendices}

\section{Appendix}
\label{sec:appendix}
\begin{prop}
\label{prop:slope_controls_Fisher}
Let $\X$ be a smooth, compact Riemannian manifold, possibly with (smooth) boundary.
Let $\H(\rho)=\int_\X \rho\log\rho $ and $\I(\rho)=\int_\X |\nabla\log\rho|^2\rho$ be the entropy and Fisher information (relatively to the volume measure $\rd x$), and denote $\W_\X$ the Wasserstein distance with quadratic cost $c(x,y)=\dist_\X^2(x,y)$.
For any $\rho\in \P(\X)$ let $(\rho_t)_{t\geq 0}$ be the heat flow started from $\rho$ (with no-flux boundary condition on $\partial\X$ if needed).
Then
\begin{equation}
\label{eq:slope_controls_Fisher}
\frac 12|\partial\H|^2(\rho)
\geq
\liminf\limits_{t\to 0}\frac{1}{t}\left[\H(\rho)-\H(\rho_t)-\frac{1}{2t}\W^2_\X(\rho,\rho_t)\right]
\geq \frac 12\I(\rho).
\end{equation}
\end{prop}
Equality is actually expected to hold:
At least formally this is clear because the heat flow is nothing but the Wasserstein gradient flow of the entropy \cite{JKO98}
\begin{equation}
\label{eq:heat_flow_grad_flow}
\frac{d\rho_t}{dt}=-\grad_{\W_\X}\H(\rho_t)
\end{equation}
and therefore
\begin{multline*}
 \frac{1}{t}\left[\H(\rho)-\H(\rho_t)-\frac{1}{2t}\W^2_\X(\rho,\rho_t)\right]
 \\
 =
 \frac{\H(\rho)-\H(\rho_t)}{t}-\frac{1}{2}\left(\frac{ \W_\X(\rho,\rho_t)}{t}\right)^2
\underset{t\to 0}{ \sim}
 -\left.\frac{d}{dt}\H(\rho_t)\right|_{t=0}-\frac 12\left\|\frac{d\rho_t}{dt}\right\|^2_{t=0}
 \\
 =-\left\langle \grad_{\W_\X}\H(\rho),\left.\frac{d\rho_t}{dt}\right|_{t=0}\right\rangle
 -\frac 12\left\|\frac{d\rho_t}{dt}\right\|^2_{t=0}
 \overset{\eqref{eq:heat_flow_grad_flow}}{=}
 \left\|\grad_{\W_\X}\H(\rho)\right\|^2-\frac 12\left\|\grad_{\W_\X}\H(\rho)\right\|^2
 \\
 =
 \frac 12 \left\|\grad_{\W_\X}\H(\rho)\right\|^2
 =\frac 12|\partial\H|^2(\rho).
\end{multline*}
We wish to stress that nothing is really new or surprising here, but since we could not find anywhere in the literature a precise statement fitting our purpose we opted for giving a self-contained proof.
This generic result is well documented in the case of $\lambda$-convex functionals, in particular equality in \eqref{eq:slope_controls_Fisher} is proved in \cite[theorem 10.4.17]{AGS} over the whole space in (possibly infinite-dimensional) Hilbert settings.
For our current choice of the entropy $\H$ this convexity is well-known to be related to Ricci curvature lower bounds on the underlying manifold.
Note that, owing to our assumption on $\X$, we always have such a uniform bound $\operatorname{Ric}_x\geq \lambda$ for some $\lambda\in\R$.
The point is that the statement is independent of the precise value of this lower bound and requires no curvature assumption on the boundary, if any.
This should be expected since all quantities in \eqref{eq:slope_controls_Fisher} are first order in nature, while curvature is a second order notion.
However, this is more subtle than meets the eye.
For, when $\X=\Omega\subset \R^d$ is a smooth, banana-shaped domain, the geodesic displacement convexity \cite{mccann1997convexity} of the entropy may completely fail ($\H$ may not be $\lambda$-convex for any $\lambda\in\R$, even very negative).
In other words, in the presence of a negatively curved boundary, the connection between Ricci lower bounds and a modulus of displacement convexity is compromised.
As a consequence the characterization of the metric slope in \cite{AGS} does not apply here due to the possible lack of convexity.
In fact this should not be a fundamental obstruction: displacement convexity is a global notion, while as already discussed our statement is local in nature (and in fact our proof below will solely rely on local PDE arguments).
Another possible interpretation is as follows:
In a purely metric context, $\lambda$-convexity can be seen as a quantitative regularity for the driving functional, and may therefore prove to be key for the abstract theory.
However we are dealing here with the explicit heat flow: the latter of course enjoys many useful properties such as regularizing effects, which in turn can be exploited in this particular context to compensate for the lack of ''metric regularity`` (the $\lambda$-convexity).
One may also hope to apply \cite[thm. 7.6]{ambrosio2014calculus}, but this requires checking a priori the lower semi-continuity of the (relaxed) metric slope, which is hard to achieve in the absence of any explicit representation (such as $|\partial\H|^2=\I$, which is precisely at stake here.)
\begin{proof}
The first inequality in \eqref{eq:slope_controls_Fisher} is trivial owing to the \emph{duality formula for the local slope} \cite[lemma 3.1.5]{AGS}, which holds in arbitrary metric spaces and reads here
$$
\frac{1}{2}|\partial\H|^2(\rho)=\limsup\limits_{\tau \downarrow 0}\frac 1\tau \sup\limits_{\nu\in \P(\X)}\left\{\H(\rho)-\H(\nu)-\frac 1{2\tau}\W_\X^2(\rho,\nu)\right\}.
$$
Thus we only focus on the second inequality.

Note first that, by usual properties of the heat flow, $\rho_t$ is bounded away from zero and $C^\infty$ smooth (up to the boundary if $\partial\X\neq\emptyset$) for any $t>0$.
The following computations and integration by parts are therefore completely legitimate as long as we step away from $t=0$.
Since $\Dom(|\partial\H|)\subset \Dom\H$, we can assume that $\H(\rho)<\infty$ and $\rho=\rho(y)\rd y$ has density w.r.t. the volume measure $\rd y$.
Let $K_t(x,y)$ be the relevant heat kernel.
Owing to the integral representation $\rho_t(x)=\int_\X K_t(x,y)\rho(y)\rd y$ and the convexity of $z\mapsto H(z)=z\log z$ we have by Jensen's inequality
\begin{multline*}
\H(\rho_t)
=
\int_\X H(\rho_t(x))\rd x
\leq
\int_\X \left(\int_\X H(\rho(y))K_t(x,y)\rd y\right)\rd x
\\
=\int_\X H(\rho(y))\left(\int_\X K_t(x,y)\rd x\right)\rd y
=\int_\X H(\rho(y))\rd y=\H(\rho).
\end{multline*}
Moreover by continuity $\rho_t\to\rho$ and lower semi-continuity of $\H$ we also have $\H(\rho)\leq \liminf\limits_{t\to 0} \H(\rho_t)$, hence in particular
 $$
 \H(\rho_t)\to \H(\rho)<+\infty
 \hspace{1cm}\mbox{as }t\downarrow 0.
 $$
Using the identity $\Delta\rho =\dive(\rho\nabla\log\rho)$ (valid at least for smooth positive $\rho$), a very classical computation gives, for $t>0$,
\begin{multline*}
\frac d{dt}\H(\rho_t)
=
\frac d{dt}\int_\X \rho_t\log\rho_t
=
\int_\X(\log\rho_t +1)\partial_t\rho_t
=
\int_\X(\log\rho_t +1)\Delta\rho_t
\\
=
\int_\X(\log\rho_t +1)\dive(\rho_t\nabla\log\rho_t)
=-\int_\X |\nabla\log\rho_t|^2\rho_t = -\I(\rho_t).
\end{multline*}
Integrating from $t_0>0$ to $t>t_0$ and taking $t_0\to 0$ with $\H(\rho_{t_0})\to \H(\rho)$, we see that $t\mapsto \I(\rho_t)$  is integrable up to $t_0=0$ with
\begin{equation}
\label{eq:estimate_Hrho_Hrho_t}
\frac{\H(\rho)-\H(\rho_t)}{t}
=
\lim\limits_{t_0\downarrow 0}\frac{\H(\rho_{t_0})-\H(\rho_t)}{t}
=
\lim\limits_{t_0\downarrow 0} \frac{1}{t}\int_{t_0}^t\I(\rho_s) \rd s
=
\frac{1}{t}\int_{0}^t\I(\rho_s) \rd s.
\end{equation}
In order to estimate the squared Wasserstein distance, we observe that $\partial_s\rho_s=\Delta\rho_s=\dive(\rho_s\nabla\log\rho_s)$, hence $(\rho_s,v_s)=(\rho_s,-\nabla\log\rho_s)$ solves the continuity equation
$$
\partial_s\rho_s+\dive(\rho_s v_s)=0
\qqtext{with}
\int_0^t \int_\X|v_s|^2\rho_s\rd s = \int_0^t\I(\rho_s)\rd s<\infty.
$$
By the Benamou-Brenier formula \cite{BB,lisini2007characterization} we get
\begin{equation}
\label{eq:estimate_W_rho_rhot}
\frac{\W_\X^2(\rho,\rho_t)}{t^2}
\leq
\frac 1t\int_0^t |v_s|^2\rho_s
=
\frac 1t\int_0^t\I(\rho_s)\rd s
\end{equation}
Gathering \eqref{eq:estimate_Hrho_Hrho_t}\eqref{eq:estimate_W_rho_rhot} we get altogether
\begin{equation*}
\forall \,t>0:\hspace{1cm}
\frac{1}{t}\left[\H(\rho)-\H(\rho_t)-\frac{1}{2t}\W^2_\X(\rho,\rho_t)\right]
\geq
\frac 1{2t}\int_0^t\I(\rho_s)\rd s
=
\frac{1}{2}\int_0^1 \I(\rho_{t u})\rd u.
\end{equation*}
Fatou's lemma finally leads to
\begin{multline*}
\liminf\limits_{t\downarrow 0}
\frac{1}{t}\left[\H(\rho)-\H(\rho_t)-\frac{1}{2t}\W^2_\X(\rho,\rho_t)\right]
\geq
\liminf\limits_{t\downarrow 0}
\frac{1}{2}\int_0^1 \I(\rho_{t u})\rd u
\\
\geq
\frac{1}{2}\int_0^1 \liminf\limits_{t\downarrow 0}\I(\rho_{t u})\rd u
\geq
\frac{1}{2}\int_0^1 \I(\rho)\rd u
=\frac{1}{2}\I(\rho),
\end{multline*}
where the last inequality follows by lower semi-continuity of $\I$ combined with continuity of the heat flow $\rho_{t u}\to \rho_0=\rho$ as $t\downarrow 0$ for a.e. $u\in (0,1)$.
Note that this also covers the case $\I(\rho)=+\infty$.
\end{proof}

\begin{prop}
\label{prop:m_boundary_tangential}
Let $\bO\subset \R^d$ be a smooth domain.
Assume that $(\rho,m)$ solves the continuity equation $\partial_t\rho+\dive m=0$ in $(0,T)\times\bO$ with no-flux boundary condition on $\pO$, in the sense that
\begin{equation}
\label{eq:CE}
\int _\bO \phi \rho_{t_1}-\int _\bO \phi \rho_{t_0} = \int_{t_0}^{t_1}\int_\bO \nabla\phi\cdot m,
\qquad
\forall \,0\leq t_0<t_1\leq T\mbox{ and }\phi\in C^1(\bO)
\end{equation}
with $t\mapsto\rho_t$ being narrowly continuous.
Then $m_\G\coloneqq m\mrest \G$ is tangential.
\end{prop}
This simply means that the no-flux boundary condition $m\cdot \nu=0$, although encoded in a weak meaning in \eqref{eq:CE}, can still be recovered in a strong sense for the boundary momentum $m_\G\cdot\nu=0$.
\begin{proof}
 For notational convenience we write $m_\O\coloneqq m\mrest \O,m_\G\coloneqq m\mrest\G$ with $m=m_\O+m_\G$, and also decompose $m_\G$ into tangential and normal components $m_\G=\mathfrak t_\G + \mathfrak n_\G $.
 We need to show that $\mathfrak n_\G=0$.

 Fix $0\leq t_0<t_1\leq T$ and take a small $\delta>0$.
 Being $\O$ a Polish space, the measure $\omega_{t_0}\in \M^+(\Omega)$ is inner regular and therefore here exists a compact set $K_\delta\subset\subset \O$ with $\o_{t_0}(\O\setminus K_\delta)\leq \delta$.
 Similarly, we can assume that $\o_{t_1}(\O\setminus K_\delta)\leq \delta$ at time $t=t_1$, and $K_\delta$ is at positive distance from $\pO$.
 Let now $Q=(t_0,t_1)\times\Omega$.
 As before, the measure $|m_\O|$ is inner regular and there exists a compact set $\tilde K_\delta\subset\subset Q$ such that $|m_{\O}|(Q\setminus \tilde K_\delta)\leq \delta$.
 By compactness $\tilde K_\delta$ is at positive distance from $\partial Q=\big([t_0,t_1]\times \pO \big)\cup \big(\{t_0\}\times \bO\big) \cup \big(\{t_1\}\times\bO\big)$.
 In particular there exists $\eta>0$ such that the smooth interior $\eta$-neighborhood $\G_\eta=\{x\in \Omega:\, \dist(x,\pO)<\eta\}$ of $\G$ satisfies $\G_\eta\subset \O\setminus K_\delta$ and $ (t_0,t_1)\times \G_\eta \subset Q\setminus \tilde K_\delta$ -- see Figure~\ref{fig:exhaustion}.
 Whence
 $$
 \o_{t_0}(\G_\eta)\leq \delta,
 \quad
 \o_{t_1}(\G_\eta)\leq \delta,
 \quad
 |m_\O|((t_0,t_1)\times\G_\eta)\leq \delta.
 $$
\begin{figure}
\centering
\def\svgwidth{8cm}
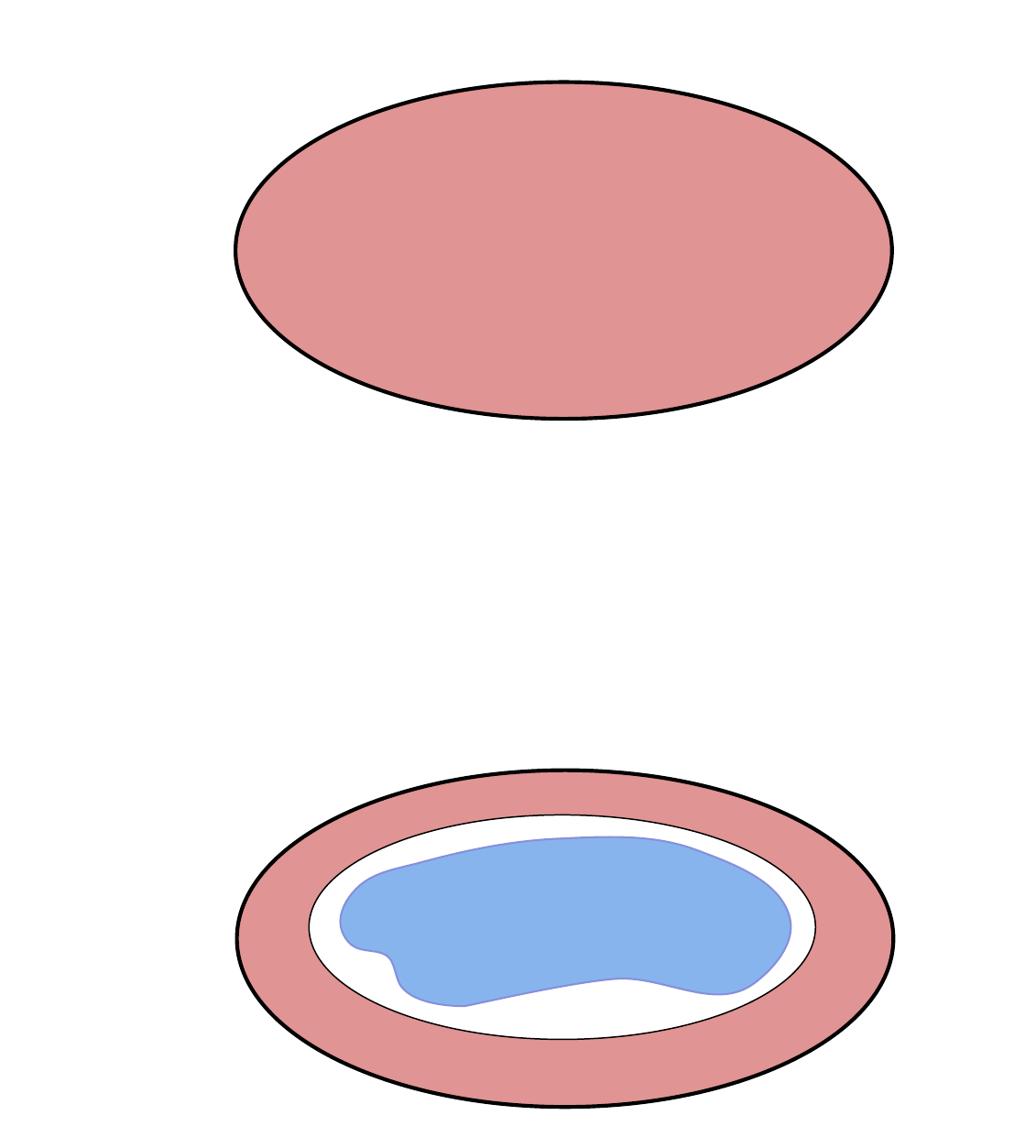
\caption{cylindrical exhaustion $(t_0,t_1)\times \G_\eta$ for $\o_{t_0},\o_{t_1},|m_\O|$}
\label{fig:exhaustion}
\end{figure}
 Pick an arbitrary test-function $\psi\in C(\G)$.
 Since $\G_\eta$ is smooth, for any small $\eta$ we can find a function $\phi^\eta(x)\in C^1(\bO)$ such that
 \begin{enumerate}[(i)]
  \item
  $\phi^\eta|_{\pO}=0$
  \item
  $\partial_\nu\phi^\eta|_{\pO}=\psi$
  \item
  $\phi^\eta$ is supported in $\G_\eta$
  \item
  we have the bounds $\|\phi^\eta\|_\infty + \|\nabla\phi ^\eta\|\leq C=C_\psi$ uniformly in $\eta\to 0$.
 \end{enumerate}
 Roughly speaking, $\phi^\eta$ is a sort of (possibly signed) bump function on the annular neighborhood  $\G_\eta$ of $\G$, vanishing on both the inner and outer boundaries, and having prescribed normal growth $\partial_\nu\phi^\eta=\psi$ across the outer boundary $\G=\pO$.
 We want to test now $\phi=\phi^\eta$ in the weak formulation of the continuity equation.
Because $\phi^\eta$ vanishes on $\pO$ and outside of $\G_\eta\subset\O$ we see that the left-hand side of \eqref{eq:CE}
 $$
   \left|\int _\bO \phi^\eta \rho_{t_1}-\int _\bO \phi^\eta \rho_{t_0}\right|
  =\left|\int _{\G_\eta} \phi^\eta \o_{t_1}-\int _{\G_\eta} \phi^\eta \o_{t_0}\right|
  \leq \|\phi^\eta\|_\infty \left(\o_{t_0}(\G_\eta) + \o_{t_1}(\G_\eta)\right)
  \leq 2 C_\psi\delta\to 0
 $$
as $\delta\to 0$.
In the right-hand side, observe that by construction $\nabla\phi^\eta$ is normal on $\G$ with slope $\partial_\nu\phi^\eta=\psi$.
Hence the tangential contribution is identically zero, and writing $\mathfrak n_\G=\nu(x) \cdot n_\G$ for some scalar measure $n_\G\in \M((0,T)\times\G)$ we get
\begin{multline*}
\int_{t_0}^{t_1}\int_\bO \nabla\phi^\eta\cdot m
=\int_{t_0}^{t_1}\int_\O \nabla\phi^\eta\cdot m_{\Omega}
+\int_{t_0}^{t_1}\int_\G \nabla\phi^\eta\cdot \mathfrak t_\G
+\int_{t_0}^{t_1}\int_\G \nabla\phi^\eta\cdot \mathfrak n_\G
\\
=\int_{t_0}^{t_1}\int_\O \nabla\phi^\eta\cdot m_{\Omega} + 0 +  \int_{t_0}^{t_1}\int_\G \partial_\nu\phi^\eta \,n_\G
=\int_{t_0}^{t_1}\int_\O \nabla\phi^\eta\cdot m_{\Omega} + \int_{t_0}^{t_1}\int_\G \psi\, n_\G.
\end{multline*}
As before the first term
$$
\left|\int_{t_1}^{t_1}\int_\O \nabla\phi^\eta\cdot m_{\Omega}\right|
=
\left|\int_{t_1}^{t_1}\int_{\G_\eta} \nabla\phi^\eta\cdot m_{\Omega}\right|
\leq
\|\nabla\phi^\eta\|_\infty |m_\Omega|((t_0,t_1)\times\G_\eta)\leq C_\psi \delta \to 0.
$$
Taking thus $\delta\to 0$ in $\int _\bO \phi^\eta \rho_{t_1}-\int _\bO \phi^\eta \rho_{t_0} = \int_{t_0}^{t_1}\int_\bO \nabla\phi^\eta\cdot m_t \rd t$ we obtain
$$
0=\int_{t_0}^{t_1}\int_\G \psi \,n_\G
$$
for all $\psi\in C(\G)$ and $t_0<t_1$.
Whence $\mathfrak n_\G=\nu(x)\cdot n_\G\equiv 0$ and the proof is complete.
\end{proof}
\begin{lem}
\label{lem:W1p_boundary_lebesgue_diff}
In $\O=\R^d_+$ let us write $x=(x_1,x')\in\R^+\times \R^{d-1}$.
 Fix $p>1$ and take $\eps,\delta\to 0$ with $\eps^{p-1}=\mathcal O(\delta^{d-1})$.
 For $x_0=(0,x_0')\in \pO$ let $\O_{\eps,\delta}=\{x=(x_1,x')\in (0,\eps) \times B_\delta^{d-1}(x_0')\}\subset \O$.
 Then for any $u\in W^{1,p}_{loc}(\bO)$ and a.e. $x_0'\in \R^{d-1}$ there holds
 $$
\fint\limits_{\Oed} |u(x) - u(x_0)|^p
=
\frac{1}{|\Oed|}\int\limits_{\Oed} |u(x) - u(x_0)|^p
\to 0
\qtext{as}\eps,\delta\to 0.
 $$
\end{lem}
Here by $u(x_0)$ we mean $[\tr u](x_0')$, the boundary trace of $u\in W^{1,p}$ evaluated at a point $x_0'$ such that $x_0=(0,x_0')$, which indeed makes sense for a.e. $x_0'\in \R^{d-1}$.
This is nothing but a Lebesgue differentiation for boundary points.
Compared to standard versions, such as \cite[thm. 5.7]{gariepy2015measure}, it is worth stressing that the shrinking sets $\Oed$ in our statement may be very thin in the $x_1$ direction and can have unbounded eccentricity (a typical requirement for standard Lebesgue differentiation).
We will typically use this with $\eps$ much smaller than $\delta$, i-e averaging on very thin boundary layers.
\begin{proof}
For simplicity let us denote $v(x')=\tr u(x')=u(0,x')$ and $u_0\coloneqq u(x_0)=v(x_0')\eqqcolon v_0$.
Assume first that $u\in C^1(\bO)$.
Then for $x=(x_1,x')\in \Oed$ we have by Jensen's inequality
$$
|u(x_1,x')-v(x')|^p
=
\left|\int_0^{x_1}\partial_{x_1} u(z,x')\rd z\right|^p
\leq
x_1^{p-1}\int_0^{x_1}|\partial_{x_1}u(z,x')|^p\rd z
\leq
x_1^{p-1}\int_0^{\eps}|\nabla u(z,x')|^p\rd z.
$$
Integrating in $x_1\in (0,\eps)$ and $x'\in B_\delta=B_\delta^{d-1}(x_0')$ we get
$$
\int_{\Oed}|u(x_1,x')-u(0,x')|^p\rd x
\leq
\frac{\eps^p}{p}\int_{\Oed}|\nabla u|^p \rd x.
$$
By triangular inequality with $(a+b)^p\leq 2^{p-1}(a^p+b^p)$ we obtain
\begin{multline*}
\frac 1{2^{p}}\int_\Oed |u(x)-u_0|^p
\leq
\int_\Oed |u(x_1,x')-v(x')|^p + \int_\Oed |v(x')-v_0|^p
\\
\leq
\frac{\eps^p}{p}\int_{\Oed}|\nabla u|^p \rd x + \eps \int_{B_\delta}|v(x')-v_0|^p.
\end{multline*}
Dividing by $|\Oed|=\eps |B^{d-1}_\delta|$ we get
$$
\frac 1{2^{p}}\fint_\Oed |u(x)-u_0|^p
\leq
\frac{\eps^p}{p |B_1^{d-1}|\eps \delta^{d-1}}\int_{\Oed}|\nabla u|^p + \fint_{B_\delta}|v(x')-v_0|^p
$$
and therefore, owing to our standing assumption that $\eps^{p-1}=\mathcal O(\delta^{d-1})$,
$$
\fint_\Oed |u(x)-u_0|^p\leq C\left(\int_{\Oed}|\nabla u|^p + \fint_{B_\delta}|v(x')-v_0|^p\right)
$$
for $C>0$ only depending on $p> 1$, the dimension $d$, and the upper bound on $\eps^{p-1}/\delta^{d-1}=\mathcal O(1)$, but not on $u$.
By density of $C^1(\bO)$ in $W^{1,p}$ and continuity of the trace operator $\tr:W^{1,p}(\O)\to L^p(\pO)$ the same holds for any $u\in W^{1,p}(\O)$.
Now since $|\nabla u|^p\in L^1$ the first term converges to zero as soon as $|\Oed|\to 0$, while the second term converges to zero simply by the Lebesgue differentiation theorem applied to $v=\tr u$ (here it is important that we chose $\Oed$ with bounded eccentricity in the $\R^{d-1}$ tangential direction).
\end{proof}
\end{appendices}
\paragraph{Acknowledgment}
J.-B.C. was supported by FCT - Funda\c{c}\~{a}o para a Ci\^encia e a Tecnologia, under the project UIDB/04561/2020.
L.M. was funded by FCT - Funda\c{c}\~{a}o para a Ci\^encia e a Tecnologia through a personal grant 2020/00162/CEECIND (DOI 10.54499/2020.00162.CEECIND/CP1595/CT0008) and project UIDB/00208/2020 (DOI 10.54499/UIDB/00208/2020).
F.S. acknowledges the support of the European Union via the ERC AdG 101054420 EYAWKAJKOS project.

\bibliographystyle{plain}
\bibliography{./biblio}

\bigskip
\noindent
{\sc
Jean-Baptiste Casteras (\href{mailto:jeanbaptiste.casteras@gmail.com}{\tt jeanbaptiste.casteras@gmail.com}).
\\
CMAFcIO, Faculdade de Ci\^encias da Universidade de Lisboa, Edificio C6, Piso 1, Campo Grande 1749-016 Lisboa, Portugal
}
\\

\noindent
{\sc
L\'eonard Monsaingeon (\href{mailto:lmonsaingeon@ciencias.ulisboa.pt}{\tt lmonsaingeon@ciencias.ulisboa.pt}).
\\
Institut \'Elie Cartan de Lorraine, Universit\'e de Lorraine, Site de Nancy B.P. 70239, F-54506 Vandoeuvre-l\`es-Nancy Cedex, France
\\
Grupo de F\'isica Matem\'atica, Departamento de Matemática, Instituto Superior T\'ecnico, Av. Rovisco Pais
1049-001 Lisboa, Portugal
}
\\

\noindent
{\sc
Filippo Santambrogio (\href{mailto:santambrogio@math.univ-lyon1.fr}{\tt santambrogio@math.univ-lyon1.fr}).
\\
Universite Claude Bernard Lyon 1, ICJ UMR5208, CNRS, Ecole Centrale de Lyon, INSA Lyon, Universit\'e Jean Monnet,
69622 Villeurbanne, France.}
\end{document}

%% file: dessin_convex.pdf_tex
\begingroup%
  \makeatletter%
  \providecommand\color[2][]{%
    \errmessage{(Inkscape) Color is used for the text in Inkscape, but the package 'color.sty' is not loaded}%
    \renewcommand\color[2][]{}%
  }%
  \providecommand\transparent[1]{%
    \errmessage{(Inkscape) Transparency is used (non-zero) for the text in Inkscape, but the package 'transparent.sty' is not loaded}%
    \renewcommand\transparent[1]{}%
  }%
  \providecommand\rotatebox[2]{#2}%
  \newcommand*\fsize{\dimexpr\f@size pt\relax}%
  \newcommand*\lineheight[1]{\fontsize{\fsize}{#1\fsize}\selectfont}%
  \ifx\svgwidth\undefined%
    \setlength{\unitlength}{231.71429011bp}%
    \ifx\svgscale\undefined%
      \relax%
    \else%
      \setlength{\unitlength}{\unitlength * \real{\svgscale}}%
    \fi%
  \else%
    \setlength{\unitlength}{\svgwidth}%
  \fi%
  \global\let\svgwidth\undefined%
  \global\let\svgscale\undefined%
  \makeatother%
  \begin{picture}(1,1.01654377)%
    \lineheight{1}%
    \setlength\tabcolsep{0pt}%
    \put(0,0){\includegraphics[width=\unitlength,page=1]{dessin_convex.pdf}}%
    \put(-0.153326453,0.4){\color[RGB]{226,46,46}\makebox(0,0)[lt]{\lineheight{1.25}\smash{\begin{tabular}[t]{l}$\rho_0=0+\g_0$\end{tabular}}}}%
    \put(0.90596204,0.4){\color[RGB]{226,46,46}\makebox(0,0)[lt]{\lineheight{1.25}\smash{\begin{tabular}[t]{l}$\rho_1=0+\g_1$\end{tabular}}}}%
    \put(0.3013816,0.15215152){\color[RGB]{46,115,226}\makebox(0,0)[lt]{\lineheight{1.25}\smash{\begin{tabular}[t]{l}$\rho_t=\o_t+0$\end{tabular}}}}%
    \put(0.5013816,0.65215152){\makebox(0,0)[lt]{\lineheight{1.25}\smash{\begin{tabular}[t]{l}$\O$\end{tabular}}}}%
    \put(0.6013816,0.75215152){\makebox(0,0)[lt]{\lineheight{1.25}\smash{\begin{tabular}[t]{l}$\G$\end{tabular}}}}%
    \put(0,0){\includegraphics[width=\unitlength,page=2]{dessin_convex.pdf}}%
  \end{picture}%
\endgroup%

%% file: dessin.pdf_tex
\begingroup%
  \makeatletter%
  \providecommand\color[2][]{%
    \errmessage{(Inkscape) Color is used for the text in Inkscape, but the package 'color.sty' is not loaded}%
    \renewcommand\color[2][]{}%
  }%
  \providecommand\transparent[1]{%
    \errmessage{(Inkscape) Transparency is used (non-zero) for the text in Inkscape, but the package 'transparent.sty' is not loaded}%
    \renewcommand\transparent[1]{}%
  }%
  \providecommand\rotatebox[2]{#2}%
  \newcommand*\fsize{\dimexpr\f@size pt\relax}%
  \newcommand*\lineheight[1]{\fontsize{\fsize}{#1\fsize}\selectfont}%
  \ifx\svgwidth\undefined%
    \setlength{\unitlength}{542.84316656bp}%
    \ifx\svgscale\undefined%
      \relax%
    \else%
      \setlength{\unitlength}{\unitlength * \real{\svgscale}}%
    \fi%
  \else%
    \setlength{\unitlength}{\svgwidth}%
  \fi%
  \global\let\svgwidth\undefined%
  \global\let\svgscale\undefined%
  \makeatother%
  \begin{picture}(1,1.10777978)%
    \lineheight{1}%
    \setlength\tabcolsep{0pt}%
    \put(0,0){\includegraphics[width=\unitlength,page=1]{dessin.pdf}}%
    \put(0.02800738,0.87147445){\color[rgb]{0,0,0}\makebox(0,0)[lt]{\lineheight{1.25}\smash{\begin{tabular}[t]{l}$t_1$\end{tabular}}}}%
    \put(0,0){\includegraphics[width=\unitlength,page=2]{dessin.pdf}}%
    \put(0.02800738,0.1780061){\color[rgb]{0,0,0}\makebox(0,0)[lt]{\lineheight{1.25}\smash{\begin{tabular}[t]{l}$t_0$\end{tabular}}}}%
    \put(0.270461132,0.1412846){\color[rgb]{0,0,0}\makebox(0,0)[lt]{\lineheight{1.25}\smash{\begin{tabular}[t]{l}$\eta$\end{tabular}}}}%
    \put(0.63235854,0.22063702){\color{blue}\makebox(0,0)[lt]{\lineheight{1.25}\smash{\begin{tabular}[t]{l}$K_\delta$\end{tabular}}}}%
    \put(0.50415035,0.48587507){\color[RGB]{44,126,48}\makebox(0,0)[lt]{\lineheight{1.25}\smash{\begin{tabular}[t]{l}$ \tilde{K}_\delta $\end{tabular}}}}%
    \put(0.66448085,0.080171069){\color{red}\makebox(0,0)[lt]{\lineheight{1.25}\smash{\begin{tabular}[t]{l}$\Gamma_\eta$\end{tabular}}}}%
    \put(0,0){\includegraphics[width=\unitlength,page=3]{dessin.pdf}}%
  \end{picture}%
\endgroup%